\documentclass{amsart}

\usepackage[colorlinks]{hyperref}
\usepackage{cancel}
\usepackage{soul}
\usepackage[normalem]{ulem}
\usepackage[all]{xy}

\usepackage{amssymb,amsmath,graphicx}
\usepackage{newtxmath}
\usepackage{mathtools}

\usepackage{xcolor}
\usepackage{enumerate}

\newtoks\prt

\newtheorem{thm}{Theorem}[section]

\newtheorem{ques}[thm]{Question}
\newtheorem{lemma}[thm]{Lemma}
\newtheorem{prop}[thm]{Proposition}
\newtheorem{cor}[thm]{Corollary}

\newtheorem{example}[thm]{Example}
\newtheorem{fact}[thm]{Fact}

\newtheorem*{claim*}{Claim}

\theoremstyle{definition}

\newtheorem{remark}[thm]{Remark}

\newtheorem{definition}[thm]{Definition}

\def\eqn#1$$#2$${\begin{equation}\label#1#2\end{equation}}

\def\1{\boldsymbol{1}}
\def\A{\mathcal A}
\def\B{\mathcal B}

\def\F{\mathcal F}

\def\W{\mathcal W}
\def\I{\mathcal I}

\def\H{\mathcal H}
\def\M{\mathcal M}

\def\P{\mathcal P}

\def\ce{\mathbb C}

\def\bnd{\operatorname{bnd}}
\def\bnd{\operatorname{bnd}}
\def\Borel{\operatorname{Borel}}
\def\co{\operatorname{conv}}
\def\ep{\varepsilon}

\def\en{\mathbb N}
\def\er{\mathbb R}
\def\qe{\mathbb Q}
\def\ef{\mathbb F}

\def\TT{S_{\ef}}

\def\wt{\widetilde}
\def\Im{\operatorname{Im}}
\def\dist{\operatorname{dist}}

\def\ov{\overline}

\def \hom {\operatorname{hom}}

\def \ahom {\operatorname{ahom}}
\def \vhom {\operatorname{{hom}}}

\def \Int {\operatorname{Int}}
\def \ext {\operatorname{ext}}

\def\span{\operatorname{span}}

\def\Rng{\operatorname{Rng}}

\def\wh{\widehat}
\def \reg {\partial _{\kern1pt\text{reg}}}

\def\inxe{X\wh{\otimes}_\ep E}

\def\di{\mbox{\,\rm d}}

\newcommand{\norm}[1]{\left\|#1\right\|}

\renewcommand{\Re}{\operatorname{Re}}

\newcommand{\abs}[1]{\left|#1\right|}
\newcommand{\setsep}{;\,}

\numberwithin{equation}{section}

\title {Operators on  injective tensor products of $L_1$-preduals}
\author{\v St\v{e}p\'{a}n Ond\v{r}ej and  Ji\v r\'\i\ Spurn\'y}

\address{Štěpán Ondřej\\
Charles University\\
Faculty of Mathematics and Physics\\
Department of Mathematical Analysis \\
Sokolovsk\'{a} 83, 186 \ 75\\Praha 8, Czech Republic}
\email{stepan.ondrej844@student.cuni.cz}

\address{Ji\v r\'\i\ Spurn\'y\\
Charles University\\
Faculty of Mathematics and Physics\\
Department of Mathematical Analysis \\
Sokolovsk\'{a} 83, 186 \ 75\\Praha 8, Czech Republic}
\email{spurny@karlin.mff.cuni.cz}



\begin{document}

\begin{abstract}
Let $X$ be an $L_1$-predual and $E,F$ be Banach spaces. We use the fact  that an unconditionally converging operator $T\colon \inxe\to F$ from the injective tensor product of $X$ and $E$ to $F$ is strongly bounded and extend $T$ to an operator $S\colon C(B_{X^*},E)\to F$ with the preservation of properties of $T$. This procedure provides a unified approach for proving properties of $\inxe$ based on the properties of $E$.
\end{abstract}

\maketitle

\tableofcontents
\clearpage

\section{Introduction}

If $E,F$ are Banach spaces and $K$ a compact Hausdorff topological space, it is well known that a bounded linear operator $T\colon C(K,E)\to F$ from the space of all continuous $E$-valued functions on $K$ to $F$ admits a representing measure $G\colon \Borel(K)\to L(E,F^{**})$ with finite semivariation $\norm{G}$, see Section~\ref{ssc:repre}. We refer the reader to \cite{batt-berg}, \cite{dobrakov}, and \cite{batt-mathan} for a detailed study of this representation. A recent treatment of this representation can be found in \cite{roth-kniha}.

The representing measure $G$ serves as a tool for proving many results on $C(K,E)$-spaces. The main feature of these proofs is the fact that an unconditionally converging (see Section~\ref{sec:oper}) operator $T\colon C(K,E)\to F$ is strongly bounded, i.e., the semivariation $\norm{G}$ of the measure $G$ is continuous at $\emptyset$, see \cite{dobrakov}. Then it also holds that $G$ has values in $L(E,F)$. This leads to a question what is the relation between properties of $T$ and $G$. This problem was extensively studied by many authors, see e.g. \cite{saab-mpcps}, \cite{abbott-atall}, \cite{abbott-lewis}, \cite{bomball-porras}, \cite{bombal-extracta}, \cite{bombal-pams}, \cite{bombal-cembranos}, \cite{bombal-salinas}, \cite{nowak}, \cite{ghenciu-popescu}, \cite{ghenciu-lewis}, \cite{lewis-brooks}, \cite{ghenciu-bullpolish}, \cite{nowak-open}. 

There is a long series of papers devoted to the question what properties of  a Banach space $E$ are transferred to the space $C(K,E)$. We mention a result in \cite{bo-ce-trans} stating that if $E$ is a Banach space with the separable dual, $C(K,E)$ has the Dieudonné property. This result was improved in \cite{kalton-saab-saab} where they showed that $C(K,E)$ possesses the Dieudonné property whenever $E$ does not contain a copy of $\ell_1$ (isomorphically).
Further, \cite{cembr-kalton-saab-saab-mathann} proved a result stating that $C(K,E)$ has the Pelczynski property (V) provided $E$ does not contain a  copy of $\ell_1$ and has the property (u). A different proof was provided in \cite{ulger} and \cite{emmanuele}. The paper \cite{randrian} gives the result that $C(K,E)$ has the Pelczynski property (V) whenever $E$ is a separable space with the property (V). It was shown in \cite{kalton-saab-saab-bsm} that $C(K,E)$ has the property semi-(V) provided $E$ has the property (u). An interested reader can consult \cite{saabs-stability} for a survey of related results. A number of results on $C(K,E)$-spaces and strongly bounded operators can be found in \cite{ghenciu-lewis}.

The space $C(K,E)$ can be considered as the space of all affine continuous $E$-valued functions defined on the convex (weak$^*$) compact set $M^1(K)$ of Borel probability regular measures on $K$. Since this set is a Bauer simplex, one might consider the problem of extending results on $C(K,E)$-spaces to the context of spaces of affine continuous $E$-valued functions defined on a compact convex set $K$ (denoted as $A(K,E)$-spaces). The first task is to establish the existence of a representing measure $G$ for an operator $T\colon A(K,E)\to F$. This was done in \cite{saab-saab}, where such a measure is constructed for $K$ being a simplex. Moreover, the representing measure is constructed in such a way that it is boundary. 
Thus we may ask whether analogous results for operators on $C(K,E)$-spaces remain valid for $A(K,E)$-spaces with $K$ being a simplex. Since $A(K,E)$ can be identified with the injective tensor product $A(K)\wh{\otimes}_\ep E$ of $A(K)$ with $E$, we may try to extend the theory even for these spaces. 
For a simplex $K$, $A(K)$ is an example of an $L_1$-predual, see \cite[Proposition 3.23]{fonf} and \cite[Chapter 7, \S 23, Theorem 6]{lacey}.
A close relation between simplices and $L_1$-preduals shows a paper \cite{lusky-comp} proving that a real $L_1$-predual is $1$-complemented in $A(K,\er)$ for a suitable simplex $K$.
A development of a representation theory for operators on injective tensor products $\inxe$, where $X$ is an $L_1$-predual, can be found in \cite{saabs-rocky-crucial}.
In particular, the following question is crucial: 

\emph{Is the representing measure for an unconditionally converging operator on $\inxe$ strongly bounded?}

The paper \cite{saab-saab-cracad} provides in Proposition~2 a positive answer, a proof appears in \cite[Proposition 1]{saabs-rocky-crucial} and \cite[Theorem 3]{saabs-illinois}.
(See also \cite{emmanuele-addendum} and \cite{emanuele-pams}  for results on injective tensor products.)

The aim of our paper is to further develop ideas from \cite{saabs-rocky-crucial}.
We extend an operator $T\colon \inxe\to F$ to an operator $S\colon C(B_{X^*},E)\to F$ with preservation of the properties of $T$. Then it is possible to deduce results for spaces $\inxe$ with the help of already known theorems for $C(K,E)$-spaces.

The paper is organized as follows. We present the necessary definitions and notions in Section~\ref{s:prel}. Then we recall the notion of the representing measure for an operator on $\inxe$ and present the known result that an unconditionally converging operator on $\inxe$ is strongly bounded. The study of strongly bounded operators is contained in Section~\ref{s:strongly}. Then we show that a strongly bounded operator on $\inxe$ admits an extension to $C(B_{X^*},E)$ with preservation of properties (Section~\ref{s:extens}). Then we provide a survey of results accessible by our extension method for the $\inxe$-spaces. The last Section~\ref{s:scatt} is devoted to a particular case of $L_1$-preduals $X$ with $\ext B_{X^*}$ scattered.

\section{Preliminaries}
\label{s:prel}
We will deal with  vector spaces over real or complex numbers, we denote the respective field by $\ef$. Our topological spaces are Hausdorff. To avoid technicalities, the compact spaces considered are nonempty and vector spaces are usually nontrivial.  If $E$ is a Banach space, $E^*$ denotes its dual, $B_E$ is the  closed unit ball of $E$, $U_E$ is the open unit ball of $E$, and $S_E$ is the unit sphere of $E$. In particular, $S_{\ef}$ is the unit sphere of the field $\ef$, i.e., $S_{\ce}=\{\lambda\in\ce\setsep \abs{\lambda}=1\}$ or $S_{\er}=\{-1,1\}$. An operator between Banach spaces means a bounded linear mapping. For an operator $T\colon E\to F$, where $E,F$ are Banach spaces, we write $T^*\colon F^*\to E^*$ for the dual operator. The notation $L(E,F)$ is reserved for the space of operators from $E$ to $F$. If $E$ is a Banach space and $A\subset E$, then $A^\perp\subset E^*$ is the annihilator of $A$. The distance between sets $A$ and $B$ in a Banach space is denoted as $\dist(A,B)$. For a measure $\mu$ and a function $f$, we sometimes write $\mu(f)$ as the shortcut to $\int f\di\mu$. The characteristic function of a set $A$ is denoted by $\chi_A$. The restriction of a function $f$ to a set $A$ is denoted as $f|_A$. The convex hull of a set $A$ in a vector space is denoted as $\co A$. The closure of a set $A$ in a topological space is denoted as $\ov{A}$, its interior as $\Int A$.

\subsection{Compact convex sets}
\label{ssc:ccs}
If $K$ is a compact topological space, we denote by $C(K,\ef)=C(K)$ the Banach space of all $\ef$-valued continuous functions on $K$ equipped with the sup-norm. This is a closed subspace of the space $B(K)=B(K,\ef)$ of all bounded Borel $\ef$-valued functions on $K$ (this space is also considered with the sup-norm). The dual of $C(K)$ will be identified (by the Riesz representation theorem) with $M(K)$, the space of all $\ef$-valued Borel regular measures on $K$ equipped with the total variation norm and the respective weak$^*$ topology. 

Let $M^+(K)$ and $M^1(K)$ stand for the set of all positive and probability measures in $M(K)$, respectively.  If $B\subset K$ is a Borel, we write $M^1(B)$ for the set of all $\mu\in M^1(K)$ with $\mu(K\setminus B)=0$. Let $\ep_t$ stand for the Dirac measure at a point $t\in K$.

Let $K$ be a compact convex set in a Hausdorff locally convex topological vector space. We write $A(K)=A(K,\ef)$ for the space of all $\ef$-valued continuous affine functions on $K$. This space is a closed subspace of $C(K)$ and is equipped with the inherited sup-norm. Given a  probability measure $\mu\in M^1(K)$, we write $r(\mu)$ for the \emph{barycenter of $\mu$}, i.e., the unique point $t\in K$ satisfying $a(t)=\int_K a \di\mu$ for each affine continuous function $a$ on $K$ (see \cite[Propositions~I.2.1 and I.2.2]{alfsen} or \cite[Chapter 7, \S\,20]{lacey}). Conversely, for a point $t\in K$, we denote by $M_{t}(K)$ the set of all  probability measures on $K$ with barycenter $t$ (i.e., the set of all probabilities \emph{representing} $t$).

The usual Choquet dilation order $\prec$ on the set $M^1(K)$ of Borel probability measures on $K$ is defined as $\mu\prec \nu$ if and only if $\mu(f)\le \nu(f)$ for any real-valued convex continuous function $f$ on $K$.   A measure $\mu\in M^1(K)$ is said to be \emph{maximal} if it is maximal with respect to the dilation order $\prec$.
In case $K$ is metrizable, maximal probability measures are exactly the probabilities carried by the $G_\delta$ set $\ext K$ of extreme points of $K$ (see, e.g., \cite[p.\,35]{alfsen} or \cite[Corollary 3.62]{lmns}). (We recall that a \emph{$G_\delta$ subset} of a topological space is a countable intersection of open sets, an \emph{$F_\sigma$ set} is a countable union of closed sets.) By the Choquet representation theorem, for any $t\in K$ there exists a maximal representing measure $\mu\in M_t(K)$ (see \cite[p.\,192, Corollary]{lacey} or \cite[Theorem I.4.8]{alfsen}). A maximal measure is carried by any $F_\sigma$ set or any Baire set containing $\ext K$, see \cite[Theorem 3.79]{lmns}. (We recall that a set $A$ in  a compact topological space $K$ is \emph{Baire}, if it is contained in the $\sigma$-algebra generated by preimages of open sets in $\er$ under continuous functions from $K$ to $\er$.)
According to the Mokobodzki test, a measure $\mu\in M^+(K)$ is maximal if and only if $\mu(g^*)=\mu(g)$ for each continuous (respectively convex continuous) function $g\in C(K,\er)$. Here 
\[
g^*(t)=\inf\{a(t)\setsep a\ge g, a\in A(K,\er)\},\quad t\in K,
\]
is the \emph{upper envelope} of $g$. We refer the reader to \cite[Chapter 7, \S 20, Theorem 2]{lacey}
for the proof of the Mokobodzki test.

A measure $\mu\in M(K)$ is called \emph{boundary} if either $\mu= 0$ or the probability measure $\frac{\abs{\mu}}{\abs{\mu}(K)}$ is maximal (here $\abs{\mu}$ is the \emph{variation} of $\mu$). If $K$ is metrizable, the boundary measures
are exactly the measures carried by $\ext K$. We write $M_{\bnd}(K)$ for the set of all boundary measure on $K$. 

A compact convex set $K$ is termed a \emph{simplex} if a maximal probability measure representing $t\in K$ is uniquely determined for each $t\in K$.
We write $\delta_t$ for this uniquely determined measure. 
Then we can define an operator $D\colon C(K)\to B(K)$ as $Df(t)=\int_K f\di\delta_t$, $t\in K$, $f\in C(K)$.
This induces a mapping $D\colon M(K)\to M_{\bnd}(K)$ by the formula $D\mu(f)=\mu(Df)$, $f\in C(K)$ (see \cite[Section 8]{kal-spu-ind}). A simplex $K$ is called a \emph{Bauer} simplex if $\ext K$ is a closed set. Then there is an isometry of $C(\ext K)$ onto $A(K)$ given (see \cite[Theorem II.4.3]{alfsen}) by $J\colon C(\ext K)\to A(K)$, where
\[
Jf(t)=\int_{\ext K} f\di\delta_t,\quad f\in C(\ext K). 
\]
Thus, all Bauer simplices are of the form $M^1(L)$ for a compact space $L$, see \cite[Corollary II.4.2]{alfsen}.

A  function $f\colon K\to \er$ is called \emph{strongly affine} if it is bounded, Borel and $\int_K f\di\mu=f(r(\mu))$, $\mu\in M^1(K)$. It is obvious from the definitions that any continuous affine function  is strongly affine, and any strongly affine function is affine.


\subsection{The dual of $C(K,E)$ and the Singer theorem}
\label{ssc:singer}

If $K$ is a compact space and $E$ is a Banach space, let $C(K,E)$ denote the space of all continuous $E$-valued functions on $K$. Then $C(K,E)$ endowed with the supremum  norm $\norm{f}=\sup_{t\in K} \norm{f(t)}_E$, $f\in C(K,E)$, is a Banach space.
By the Singer theorem (see \cite{hensgen} for an easy proof) the dual space $C(K,E)^*$ of $C(K,E)$ is canonically isometric to $M(K,E^*)$, the space 
of all countably additive regular $E^*$-valued Borel measures on $K$ with bounded variation endowed with the variation norm $\norm{\mu}=\abs{\mu}(K)$.

Let $\Borel(K)$ denote the Borel $\sigma$-algebra on a compact space $K$.
We recall that a mapping $\mu\colon \Borel(K)\to F$, where $F$ is a Banach space, is a \emph{Borel vector measure} if it is finitely additive on $\Borel(K)$ (see \cite[p.\,1, Definition 1]{diesteluhl}). If $\mu$ is countably additive, $\mu$ is called a \emph{countably additive measure}. A measure $\mu$ is \emph{bounded}, if the range of $\mu$ is bounded in $F$.
The \emph{variation} $\abs{\mu}$ of $\mu$ is defined as 
\[
\abs{\mu}(B)=\sup_{\pi(B)}\sum_{A\in \pi(B)}\norm{\mu(A)},\quad B\in\Borel(K),
\]
(here $\pi(B)$ denotes the set of all finite Borel decompositions  of the set $B\in\Borel (K)$).
A vector measure $\mu\colon \Borel(K)\to F$ is \emph{of bounded variation} if $\abs{\mu}(K)<\infty$. 

If $\mu\colon \Borel(K)\to E^*$ of the bounded variation is countably additive, then $\abs{\mu}$ is a countably additive (see \cite[p.\,3, Proposition 9]{diesteluhl}) Borel measure.  A countably additive Borel measure $\mu\colon \Borel(K)\to E^*$ of bounded variation is \emph{regular} provided the scalar measure $\abs{\mu}$ is regular, i.e., it is inner regular with respect to closed sets. The space $M(K,E^*)$ consists of such measures $\mu$.

Given $\mu\in M(K,E^*)$ and $x\in E$, we denote by $\mu_x$ the measure $\mu_x\colon \Borel(K)\to\ef$ given by $\mu_x(B)=\mu(B)x$, $B\in\Borel(K)$.

Let us recall the definition of the integral $\int f\di\mu$ for $\mu\in M(K,E^*)$ and $f\in C(K,E)$.
If $\mu\in M(K,E^*)$ and $f=\sum_{j=1}^n \chi_{A_j}\cdot x_j$ is an $E$-valued simple Borel function 
(i.e., $x_1,\dots,x_n\in E$ and $A_1,\dots,A_n$ are pairwise disjoint Borel subsets of $K$), then
$$\int f\di\mu=\sum_{j=1}^n \mu(A_j)(x_j).$$
It is easy to check that $f\mapsto \int f\di\mu$ is a linear functional on the space of $E$-valued
simple Borel functions and $\abs{\int f\di\mu}\le \norm{f}_\infty\abs{\mu}(K)$. Hence, it may be uniquely continuously extended to the space of $E$-valued functions which may be uniformly approximated by $E$-valued simple Borel functions. We denote this space by $B(K,E)$. It clearly contains $C(K,E)$. (Note that in case $E=\er$ this integral coincides with the classical Lebesgue integral.)

If $\varphi\colon K\to K'$ is a continuous surjection of a compact space onto a compact space and $E$ is a Banach space, $\varphi$ induces a weak$^*$-to-weak$^*$ continuous mapping (denoted likewise) $\varphi\colon M(K,E^*)\to M(K',E^*)$ simply by taking $(\varphi \mu)(f')=\mu(f'\circ \varphi)$, $\mu\in M(K,E^*)$, $f'\in C(K',E)$. Then we have $\abs{\varphi\mu}\le \varphi(\abs{\mu})$.

\subsection{The representing measures for operators on $C(K,E)$}
\label{ssc:repre}

Let $K$ be a  compact space and let $E,F$ be Banach spaces. 
Let $G\colon \Borel(K)\to L(E,F^{**})$ be a bounded vector measure. Then the \emph{semivariation} $\norm{G}$ of $G$ is defined as
\[
\norm{G}(B)=\sup\left\{\norm{\sum_{i=1}^n G(E_i)x_i}\setsep (E_i)_{i=1}^n\subset B\text{ pairwise disjoint Borel}, (x_i)_{i=1}^n\subset B_E, n\in\en\right\},
\]
where $B\in \Borel (K)$.
For each $y^*\in F^*$, let $G_{y^*}\colon \Borel (K)\to E^*$ be defined as $G_{y^*}(B)x=(G(B)x)(y^*)$, $x\in E$. 
Then we have the following formula (see \cite[p.\,55, Proposition 5]{dinculeanu}):
\begin{equation}
    \label{eq:semivar}
    \norm{G}(B)=\sup_{y^*\in B_{F^*}}\abs{G_{y^*}}(B),\quad B\in \Borel(K).
\end{equation}

Assume that $\norm{G}(K)<\infty$, i.e., $G$ is of \emph{bounded semivariation}. Then the formula $Uf=\int f\di G$, $f\in B(K,E)$, defines an operator $U\colon B(K,E)\to F^{**}$ such that $\norm{U}=\norm{G}(K)$.
For a proof of this result, see \cite[Theorem 3.3]{mu-oja-pi}.

We note that the integral $\int_K f\di G$ is obtained as the extension of the linear mapping $f\mapsto \int f \di G$, where $f$ runs through simple functions in $B(K,E)$.

Indeed, for $f=\sum_{i=1}^n \chi_{E_i}x_i$, where $(x_i)\subset E$ and $(E_i)\in\pi(K)$ is a Borel decomposition of $K$, we have
\[
\int_K f\di G=\sum_{i=1}^n G(E_i)(x_i).
\]
Hence, for $y^*\in B_{F^*}$ we have
\[
\begin{aligned}
\abs{\left(\int f\di G\right)y^*}&=\abs{\sum_{i=1}^n(G(E_i)x_i)(y^*)}= \abs{\sum_{i=1}^n G_{y^*}(E_i)(x_i)}\le \sum_{i=1}^n \norm{G_{y^*}(E_i)}_{E^*}\norm{x_i}_E\\
&\le \sum_{i=1}^n\norm{G_{y^*}(E_i)}_{E^*}\norm{f}_\infty\le \abs{G_{y^*}}(K)\norm{f}_\infty\le \norm{G}(K)\norm{f}_\infty.
\end{aligned}
\]
Hence 
\begin{equation}
\label{eq:integral}
\norm{\int f\di G}_{F^{**}}\le \norm{G}(K)\norm{f}_\infty,
\end{equation}
and thus  the integral defines a bounded operator on simple functions in  $B(K,E)$ that has values in $F^{**}$, and hence admits a unique extension to the operator $U\colon B(K,E)\to F^{**}$. 

If $f$ is a function in $B(K,E)$ and $B\in \Borel(K)$, then $g(t)=\begin{cases} f(t),& t\in B,\\
                                                       0,&t\in K\setminus B,\end{cases}$
is also in $B(K,E)$ and thus $\int_B f\di G$ can be understood as $\int_K g\di G$.

If $G\colon \Borel(K)\to L(E,F^{**})$ is a bounded vector measure and $y^*\in F^{*}, x\in E$, we denote by $G_{y^*,x}\colon \Borel(K)\to \ef$ the finitely additive scalar measure defined as
\[
G_{y^*,x}(B)=(G(B)x)(y^*),\quad B\in\Borel(K).
\]
If $\norm{G}(K)<\infty$, $f\in B(K)$,  $x\in E$ and $y^*\in F^*$, then, if $G_{y^*}\in M(K,E^*)$, we have 
\[
\int_{K} f\di G_{y^*,x}=\int_K (f\otimes x)\di G_{y^*}.
\]

The measure $G$ is called \emph{strongly bounded} if its semivariation is continuous at $\emptyset$, i.e., if for every decreasing sequence of Borel sets $(B_i)$ with $\bigcap_{i=1}^\infty B_i=\emptyset$ we have $\lim_{i\to\infty}\norm{G}(B_i)=0$. We have the following characterization of strongly bounded measures.

\begin{thm}
    \label{t:sbmeasures}
 Let $K$ be a compact  space and let $E,F$ be Banach spaces. Let $G\colon \Borel(K)\to L(E,F^{**})$ be of bounded semivariation such that for every $y^*\in F^*$ the vector measure $G_{y^*}$ is a regular countably additive $E^*$-valued measure, i.e., $G_{y^*}\in M(K,E^*)$. Then the following assertions are equivalent:
 \begin{enumerate}[(i)]
 \item The measure $G$ is strongly bounded,
 \item The set $\{\abs{G_{y^*}}\setsep y^*\in B_{F^*}\}$ is relatively weakly compact in $M(K)$.
 \item The set $\{\abs{G_{y^*}}\setsep y^*\in B_{F^*}\}$ is uniformly countably additive.
 \item There exists a measure $\lambda\in M^+(K)$ such that 
 \[
 \lim_{B\in \Borel(K),\lambda(B)\to 0}\norm{G}(B)=0,
 \]
 i.e., we have that
 \[
 \forall \ep>0\exists \delta>0\forall B\in\Borel(K), \lambda(B)<\delta\colon \norm{G}(B)<\ep.
 \]
  \end{enumerate}
 \end{thm}

\begin{proof}
 Conditions (ii), (iii) are equivalent due to \cite[Theorem IV.9.1, p.\,305]{dunford-schwartz}. Assertions (ii) and (iv) are equivalent by \cite[Theorem IV.9.2, p.\,306]{dunford-schwartz}. Clearly, (iv)$\implies$(i). The implication (i)$\implies$(iii) then follows from formula~\eqref{eq:semivar}.     
\end{proof}

Now we can state the representation theorem for operators on $C(K,E)$.
For the proof, see \cite[Corollary 4.10]{mu-oja-pi}.

\begin{thm}
    \label{t:riesz-singer}
Let $K$ be a compact space, and let $E,F$ be Banach spaces.     
Let $U\colon C(K,E)\to F$ be an operator. Then there exists a unique vector measure $G\colon \Borel(K)\to L(E,F^{**})$ with  bounded semivariation such that:
\begin{itemize}
    \item $Uf=\int f\di G$, $f\in C(K,E)$;
    \item $\norm{U}=\norm{G}$;
    \item for each $y^{*}\in F^{*}$, we have $G_{y^*}\in M(K,E^*)$ and $G_{y^*}=U^*y^*$.
\end{itemize}    
\end{thm}

We call $G$ the \emph{representing measure} of the operator $U$. 
 An operator $U$ is \emph{strongly bounded} if  
its representing measure is strongly bounded.
We have the following information on strongly bounded operators.

\begin{thm}
    \label{t:sbo}
    Let $K$ be a compact space and let $E,F$ be Banach spaces.     
Let $U\colon C(K,E)\to F$ be a strongly bounded operator with the representing measure $G$. 
\begin{itemize}
\item Then $G$ has values in $L(E,F)$.
\item There exists a measure $\lambda\in M^+(K)$ such that 
\[
\lambda(B)\le \norm{G}(B),\quad B\in\Borel(K),\qquad\text{and}\qquad\lim_{B\in\Borel(K), \lambda(B)\to 0}\norm{G}(B)=0.
\]
\end{itemize}
\end{thm}

\begin{proof}
For the proof of the first assertion, see \cite[Corollary 4.4.1]{lewis-brooks}, and the proof of the second assertion can be found in \cite[Lemma 2]{dobrakov}.
\end{proof}

We mention the following frequently used observation useful in various estimates.

\begin{fact}
    \label{f:odhad}
Let $K$ be a compact space and $E,F$ be Banach space.  
Let $G\colon \Borel(K)\to L(E,F^{**})$ be of bounded semivariation such that for every $y^*\in F^*$ the vector measure $G_{y^*}$ is a regular countably additive $E^*$-vector valued measure, i.e., $G_{y^*}\in M(K,E^*)$.
 If $f\in B(K,E)$ is bounded by $M\ge 0$ on a set $B\in \Borel(K)$, then
\[
\norm{\int_{B} f\di G}_{F^{**}}\le M\cdot \norm{G}(B).
\]    
\end{fact}

\begin{proof}
First, we observe that $\norm{\int_K f\di G}_{F^{**}}\le \norm{f}_\infty \norm{G}(K)$ for every $f\in B(K,E)$ due to \eqref{eq:integral}.

Given $f\in B(K,E)$ bounded by $M$ on $B\in\Borel(K)$,  let $g(t)=\begin{cases} f(t),& t\in B,\\
                                        0,&t\in K\setminus B,\end{cases}$.
Let $\ep>0$ be given. Then there exists a function $h=\sum_{i=1}^n \chi_{E_i}x_i$, where $(E_i)$ is a Borel decomposition of $K$ and $(x_i)\subset E$ such that
\begin{itemize}
    \item $\sup_{t\in K}\norm{g(t)-h(t)}_E<\ep$,
    \item $B$ respects $(E_i)_{i=1}^n$, i.e., for each $i\in\{1,\dots,n\}$ either $E_i\subset B$ or $E_i\cap B=\emptyset$.
 \end{itemize}   
Then 
\[
\norm{h(t)}_E\le \begin{cases} \ep,& t\in K\setminus B,\\
               M+\ep,&t\in B.\end{cases}
               \]
Let 
\[
h_1(t)=\begin{cases} h(t),& t\in K\setminus B,\\
                            0,&t\in B,
                            \end{cases}\quad\text{and}\quad
h_2(t)=\begin{cases} 0,& t\in K\setminus B,\\
                            h(t),&t\in B.                           
        \end{cases}                                                
\]
If $I\subset \{1,\dots,n\}$ satisfies $B=\bigcup_{i\in I} E_i$, then $h_2(t)=\sum_{i\in I}\chi_{E_i}(t)x_i$.
Hence,
\[
\begin{aligned}
\norm{\int_B f\di G}_{F^{**}}&=\norm{\int_K g\di G}_{F^{**}}=\norm{\int_K (g-h)\di G+\int_K h\di G}_{F^{**}}\\
&\le \norm{\int_K (g-h)\di G}_{F^{**}}+\norm{\int_K h_1\di G}_{F^{**}}+(M+\ep)\norm{\sum_{i\in I} G(E_i)(\frac{x_i}{M+\ep})}_{F^{**}}\\
&\le \ep\norm{G}(K)+\ep\norm{G}(K)+(M+\ep)\norm{G}(B).
\end{aligned}
\]
As $\ep>0$ is arbitrary, the claim follows.
\end{proof}

\subsection{Operators on Banach spaces}
\label{sec:oper}

If $T\colon E\to F$ is an operator between Banach spaces $E$ and $F$, we can study its (sequential) behavior with respect to different topologies on respective spaces. These considerations lead to the  notions below.

But first we recall a slightly less usual definition of the Right topology.
If $E$ is a Banach space, let $\rho$ be the topology of uniform convergence on weakly compact subsets of $E^*$. This topology is called the \emph{Right topology} and  is the restriction to $E$ of the Mackey topology $\mu(E^{**}, E^*)$ on $E^{**}$ with respect to the dual pair $(E^{**}, E^{*})$. (For more details, see \cite{pe-vi-wr-y}.)

Let $U\colon E\to F$ be an operator on a Banach space $E$ with values in a Banach space $F$. We say that $U$ is 
\begin{itemize}
\item  \emph{compact} (c) if it maps the unit ball to a relatively compact set;
   \item \emph{weakly compact} (wc) if it maps the unit ball to a relatively weakly compact set;
    \item \emph{completely continuous} (cc) if it is weak-to-norm sequentially continuous;
    \item \emph{weakly precompact} (wpc) if the image of any bounded sequence admits a weakly Cauchy subsequence;
    \item \emph{pseudo weakly compact} (pwc) if it is Right-to-norm sequentially continuous;
    \item \emph{weakly completely continuous} (wcc)  if it maps weakly Cauchy sequences to weakly convergent sequences;
    \item \emph{Right completely continuous} (Rcc) if it maps Right Cauchy sequences to Right convergent sequences;
    \item \emph{unconditionally converging} (uc) if it maps weakly unconditionally Cauchy series to the unconditionally convergent series (we recall that a series $\sum_{n=1}^\infty x_n$ in a Banach space $E$ is \emph{weakly unconditionally Cauchy} if for all $x^*\in E^*$ the series $\sum_{n=1}^\infty\abs{x^*(x_n)}$ converges and it is \emph{unconditionally convergent} if the series $\sum_{n=1}^\infty t_n x_n$ converges whenever $(t_n)$ is a bounded sequence of scalars).
\end{itemize}

It follows from \cite{kacena-jmaa}, \cite{pe-vi-wr-y} and \cite{krulisova} that for every operator $U$ between Banach spaces the following implications hold:
\begin{equation}
    \label{eq:implic}
\begin{array}[t]{ccccccccc}
&&&&\text{wpc}&&&&\\
&&&\Nearrow&&&&&\\
              &&\text{wc}&\implies&\text{pwc}&&&& \\
              &\Nearrow&&&&\Searrow\\
              
            \text{compact}&&&\mathrlap{\Searrow}{\Nearrow}&&& \text{Rcc}&\implies& \text{uc.}\\
            &\Searrow&&&&\Nearrow&&\\
            &&\text{cc}&\implies&\text{wcc}&&
        \end{array}
\end{equation}

For a Banach space $E$, the previous concepts inspire the definitions below. 
For a Banach space $E$, we say that
\begin{itemize}
\item $E$ has the \emph{Pelczy\'nski property} (V) if every unconditionally converging operator on $E$ is weakly compact;
\item  $E$ has the \emph{property semi-(V)} if every unconditionally converging operator on $E$ is weakly completely continuous;
\item $E$ has the \emph{Dunford-Pettis property} (DPP) if  every weakly compact operator on $E$ is completely continuous;
\item $E$ has the \emph{reciprocal Dunford-Pettis property} (RDP) if every  completely continuous operator on $E$  is weakly compact;
\item $E$ has the \emph{Dieudonné property} (D) if every weakly completely continuous operator on $E$ is weakly compact;
\item $E$ has the \emph{Right Dieudonné property} (RD) if every Right completely continuous  operator on $E$ is weakly compact;
\item $E$ is \emph{sequentially Right} (SR) if every pseudo weakly compact operator on $E$ is weakly compact;
\item  $E$ has the \emph{Schur property} (S) if every weakly null sequence is norm null;
\item $E$ has the \emph{property (u)}  if  for every weakly Cauchy sequence $(x_n)$ in $E$, there exists a weakly unconditionally Cauchy series $\sum_{n=1}^\infty z_n$ in $E$ such that the sequence $(x_n - \sum_{j=1}^n z_j)$ is weakly null.
\end{itemize}

Then, for a Banach space $E$, we have 
\[
E\text{ has (S)}\implies E\text{ has (DPP)},
\]
and   from \eqref{eq:implic} (see also \cite{krulisova}) we get the following diagram:

\begin{equation}
    \label{eq:spaces}
\begin{array}[t]{ccccccc}
              &&&&E\text{ is (SR)}&&\\
              &&&\Nearrow&&\Searrow&\\
              E\text{ has (V)}&\implies& E\text{ has (RD)}&\implies&E \text{ has (D)}&\implies& E\text{ has (RDP)}\\
            &\Searrow&&&&&\\
           E\text{ has (u)}&\implies& E\text{ has semi-(V).}&&&&
        \end{array}
\end{equation}

\subsection{$L_1$-preduals and their injective tensor products}
\label{ss:preduals}

If $K$ is a compact convex set and $E$ is a Banach space, let  $A(K,E)$ stand for the space of all affine continuous functions on $K$ with values in $E$, where the norm is induced from the space $C(K,E)$.  A measure $\mu\in M(K,E^*)$ is called \emph{boundary} if its variation $\mu$ is maximal on $K$. We write $M_{\bnd}(K,E^*)$ for the set of all boundary measures $\mu\in M(K,E^*)$.

If $X$ is a Banach space, a set $A\subset B_{X^*}$ is \emph{$S_{\ef}$-homogeneous} (or just \emph{homogeneous}), if $\alpha A=A$ for each $\alpha\in S_{\ef}$.
If $X, E$ are Banach spaces over $\ef$ and $f\colon B_{X^*}\to E$ is a function, it is called \emph{$S_{\ef}$-homogeneous} (or \emph{homogeneous}) if $f(\alpha t)=\alpha f(t)$, $t\in B_{X^*}$, $\alpha\in \TT$. Hence a homogeneous function is \emph{odd} in case $\ef=\er$. We write $C_{\hom}(B_{X^*},E)$ for the space of all homogeneous continuous functions from $B_{X^*}$ into $E$.
A measure $\mu\in M(B_{X^*},E^*)$ is \emph{antihomogeneous} if $\mu(\alpha B)=\ov{\alpha} \mu(B)$ for each $B\in \Borel(B_{X^*})$ and $\alpha\in\TT$. We write $M_{\ahom}(B_{X^*}, E^*)$ for the antihomogeneous elements of $M(B_{X^*},E^*)$. A measure $\mu\in M(B_{X^*},E^*)$ is called \emph{invariant} if $\mu(\alpha B)=\mu(B)$ for each $B\in\Borel(B_{X^*})$ and $\alpha\in\TT$. The notation $M_{\bnd,\ahom}(B_{X^*},E^*)$ is reserved for boundary antihomogeneous elements of $M(B_{X^*},E^*)$.

Let $\di \alpha$ denote the probability Haar measure on $\TT$. Then we consider an operator $\vhom\colon C(B_{X^*}, E)\to C_{\hom}(B_{X^*},E)$ defined as
\[
\vhom f(t)=\int_{\TT} \alpha^{-1}f(\alpha t)\di\alpha,\quad  t\in B_{X^*}, f\in C(B_{X^*},E).
\]
(The integral is the Bochner one - it clearly exists as the integrand is a continuous vector function on the compact space $\TT$. Also, it easily follows that $\hom f\in C_{\hom}(B_{X^*},E)$ for each $f\in C(B_{X^*},E)$, see \cite[p.\,524]{saabs-illinois}.)

Then we can consider the dual operator $\vhom \colon M(B_{X^*},E^*)\to M(B_{X^*}, E^*)$ defined  as
\[
\vhom \mu(f)=\int_{B_{X^*}}\vhom f\di\mu,\quad f\in C(B_{X^*},E),\mu\in M(B_{X^*}, E^*).
\]

We have the following fact.
\begin{fact}
    \label{f:homog}
   We have  $\vhom \mu\in M_{\ahom}(B_{X^*},E^*)$ for each $\mu\in M(B_{X^*},E^*)$.
\end{fact}
\begin{proof}
To see that $\vhom \mu$ is an antihomogeneous measure, consider any $B\in\Borel(B_{X^*})$ and $\alpha\in\TT$. Given $x\in X$, we have $(\hom \mu)_x=\hom(\mu_x)$. 

Indeed, for $f\in C(B_{X^*})$ we have 
\[
\begin{aligned}
(\hom \mu)_x(f)&=\int_{B_{X^*}} (f\otimes x)\di\hom \mu=\int_{B_{X^*}} (\hom f\otimes x)\di\mu\\
&=\int_{B_{X^*}}\hom f \di(\mu_x)=(\hom \mu_x)(f).
\end{aligned}
\]
This gives that it is enough to consider $\mu\in M(B_{X^*})$. But then the equality $(\hom \mu)(\alpha B)=\ov{\alpha} (\hom \mu)(B)$ follows from \cite[Lemma 3.17(a)]{Kaspu-studia}. 
\end{proof}

We mention that antihomogeneous measures are called homogeneous in \cite{saabs-illinois} and \cite{effros}.

Let $K$ be a compact convex set and $X=A(K)$. Let $\phi\colon K\to B_{X^*}$ be the evaluation mapping given by $\phi(t)(a)=a(t)$, $a\in X$, $t\in K$. Then we have a homeomorphic mapping $\theta\colon \TT\times K\to B_{X^*}$ given as
$\theta(\alpha, t)=\alpha\phi(t)$, $(\alpha,t)\in \TT\times K$. Since $\TT\cdot \phi(K)$ is a compact space containing extreme points of $\ext B_{X^*}$, any boundary measure $\mu\in M(B_{X^*},E^*)$ is carried by $\theta(\TT\times K)$.

For a measure $\omega\in M(\TT\times K,E^*)$ we denote by $H\omega$ the measure in $M(K,E^*)$ defined as
\[
H\omega(f)=\int_{\TT\times K}\alpha f(t)\di\omega(\alpha,t),\quad f\in C(K,E).
\]
\begin{lemma}
\label{l:prenos}
Let $K$ be a compact convex set, $X=A(K)$ and $E$ be  a Banach space.
\begin{enumerate}[(a)]
\item $\mu\in M(K,E^*)$ is boundary if and only if $\mu_x\in M(K)$ is boundary for any $x\in E$.
    \item If $\mu\in M_{\bnd}(K,E^*)$, then $\phi(\mu)$ and $\vhom \phi(\mu)$ are boundary on $B_{X^*}$.
    \item  If $\nu\in M_{\bnd}(B_{X^*},E^*)$, then $H(\theta^{-1}(\nu))\in M_{\bnd}(K,E^*)$.
\end{enumerate}
\end{lemma}

\begin{proof}
(a) If $\mu\in M(K,E^*)$ is boundary and $x\in E$, then $\abs{\mu_x}\le \abs{\mu}\norm{x}$ is maximal, and thus $\mu_x$ is boundary.

Conversely, if $\mu\in M(K,E^*)$ is not boundary, then there exists $g\in C(K,\er)$ convex such that for $B=\{t\in K\setsep g^*(t)=g(t)\}$ we have $\abs{\mu}(K\setminus B)>0$. Let $A\subset K\setminus B$ be Borel such that $\norm{\mu(A)}>0$. Then there exists $x\in E$ with $\abs{\mu(A)x}>0$.
But then $\abs{\mu_x}(A)>0$, and hence $\abs{\mu_x}$ is not maximal. Hence also $\mu_x$ is not boundary.

(b) It is easy to see that, for any $\mu\in M(K,E^*)$ and $x\in E$, 
\[
(\vhom \phi(\mu))_x=\hom\phi(\mu_x).
\]
Hence it is enough to show the assertion for a scalar measure $\mu$. But then $\phi\mu$ is boundary since $\phi\colon K\to B_{X^*}$ is an affine homeomorphism of $K$ onto the closed face $\phi(K)$ of $B_{X^*}$. Further, $\hom \phi(\mu)$ is also boundary by \cite[Chapter 7, \S 23, Lemma 10]{lacey}.

(c) As above we observe that $H(\theta^{-1}(\nu))_x=H(\theta^{-1}(\nu_x))$, $x\in E$, and thus it is enough to check the assertion for a scalar measure $\nu$. But this follows from \cite[Lemma 3.6(v)]{ka-spu-constants}.
\end{proof}

For the dual ball $B_{X^*}$ of a Banach space $X$ and a Banach space $E$, we write $A_{\hom}(B_{X^*},E)$ for the space of all affine homogeneous continuous functions from $B_{X^*}$ to $E$.

\begin{fact}
    \label{f:dualita}
    Let $X$ be a Banach space and $f\colon B_{X^*}\to \ef$. Then the following assertions are equivalent.
    \begin{enumerate}[(i)]
        \item There exists $x\in X$ such that $f(t)=t(x)$, $t\in B_{X^*}$.
        \item The function $f$ belongs to $A_{\hom}(B_{X^*},\ef)$.
    \end{enumerate}
\end{fact}

\begin{proof}
Obviously, (i)$\implies$(ii). The converse follows from the fact that any $f\in A_{\hom}(B_{X^*},\ef)$ admits a (unique) linear extension $\wt{f}\colon X^*\to \ef$. Since $\wt{f}$ is continuous on any $rB_{X^*}$, $r>0$, $\wt{f}$ is weak* continuous on $X^*$. Hence, there exists $x\in X$ such that $\wt{f}(t)=t(x)$, $t\in X^*$.    
\end{proof}

\begin{lemma}
\label{l:approx-teleman}
Let $X$ be a Banach space and $f\colon B_{X^*}\to \ef$ be a Borel homogeneous strongly affine function. Then for any $\mu\in M^1(B_{X^*})$ there exists a sequence $(x_n)\subset X$ such that $x_n\to f$ $\mu$-almost everywhere and $\norm{x_n}\le \norm{f}$, $n\in\en$.
\end{lemma}

\begin{proof}
If $X$ is a real Banach space, the assertion follows from the proof of \cite[Theorem 2.3(i)]{teleman-increst} (see also \cite[Theorem 4.34]{lmns}). Indeed, we consider a measure $\nu=\mu+\sigma_{-1}\mu$, where $\sigma_{-1}t=-t$, $t\in B_{X^*}$. By the cited theorem, for this measure we can find a sequence $(f_n)$ in $A(B_{X^*},\er)$ such that $\norm{f_n}\le \norm{f}$ and $f_n(t)\to f(t)$ for $t\in B_{X^*}\setminus A$, where $A\subset B_{X^*}$ is Borel with $\nu(A)=0$.
Then
\[
0=(\sigma_{-1}\mu)(A)=\mu(\sigma_{-1}^{-1}(A))=\mu(-A),
\]
and thus 
\[
B=B_{X^*}\setminus (A\cup -A)
\] 
is a Borel set with $\mu(B)=1$. Let $g_n=\frac12(f_n(t)-f_n(-t))$, $t\in B_{X^*}$, $n\in\en$. Then $g_n$ is a continuous odd (i.e., $S_{\er}$-homogeneous) affine function on $B_{X^*}$ with $\norm{g_n}\le \norm{f}$. Moreover, for $t\in B$ we have $t,-t\notin A$, and thus 
\[
g_n(t)=\frac12(f_n(t)-f_n(-t))\to \frac12 (f(t)-f(-t))=f(t).
\]
By Fact~\ref{f:dualita}, there exists a sequence $(x_n)\subset X$ such that $g_n(t)=t(x_n)$, $t\in B_{X^*}$. Then $(x_n)$ is the desired sequence.

Let us now consider the case of a complex Banach space.
Let $f=f_1+if_2$, where $f_1=\Re f$ and $f_2=\Im f$. Then $f_1, f_2$  are strongly affine real functions with $f_1(0)=f_2(0)=0$ and $f_2(t)=-f_1(it)$, $t\in B_{X^*}$. Given $\mu\in M^1(B_{X^*})$, let $\nu=\mu+\sigma_{-1}\mu+\sigma_i\mu+\sigma_{-i}\mu$ (here $\sigma_\lambda(t)=\lambda t$, $t\in B_{X^*}$, $\lambda\in \{-1,i,-i\}$). By \cite[Theorem 2.3]{teleman-increst}, there exists a sequence $(f_n)\subset A(B_{X^*},\er)$ such that $f_n(t)\to f_1(t)$ for $t\in B_{X^*}\setminus A$, where $A\subset B_{X^*}$ is a Borel set with $\nu(A)=0$, and $\norm{f_n}\le \norm{f_1}$.

Then 
\[
0=(\sigma_i\mu)(A)=\mu(\sigma_i^{-1}(A))=\mu(i^{-1}A),
\]
\[
0=(\sigma_{-1}\mu)(A)=\mu(\sigma_{-1}^{-1}(A))=\mu(-A).
\]
and
\[
0=(\sigma_{-i}\mu)(A)=\mu(\sigma_{-i}^{-1}(A))=\mu(iA).
\]
Let a Borel set $B$ be defined as
\[
B=B_{X^*}\setminus (A\cup -A\cup i^{-1}A\cup iA),
\]
then $\mu(B)=1$. 

We consider the function $g_n(t)=\frac12(f_n(t)-f_n(-t))$, $t\in B_{X^*}$, $n\in\en$. Then $g_n\colon B_{X^*}\to \er$ are affine, continuous, and odd.
Then $\norm{g_n}\le \norm{f_n}\le \norm{f_1}$ and for $t\in B$, $-t, t\notin A$, and thus  
\[
g_n(t)=\frac12(f_n(t)-f_n(-t))\to \frac12 (f_1(t)-f_1(-t))=f_1(t).
\]
Furthermore, if $t\in B$, then $it, -it\notin A$, and thus 
\[
g_n(it)=\frac12(f_n(t)-f_n(-it))\to \frac12 (f_1(it)-f_1(-it))=f_1(it).
\]
Hence, it follows that the functions $a_n(t)=g_n(t)-ig_n(it)$, $t\in B_{X^*}$, $n\in\en$, are affine, homogeneous, continuous and, for $t\in B$ satisfy
\[
a_n(t)=g_n(t)-ig_n(it)\to f_1(t)-if_1(it)=f_1(t)+i f_2(t)=f(t).
\]
By Fact~\ref{f:dualita}, there exist elements $x_n\in X$, $n\in\en$, such that $a_n(t)=t(x_n)$, $t\in B_{X^*}$. Finally, for $n\in\en$, we have by the linearity of $a_n$ 
\[
\norm{x_n}=\sup_{t\in B_{X^*}}\abs{a_n(t)}=\sup_{t\in B_{X^*}}\abs{g_n(t)}\le \norm{f_1}\le \norm{f}.
\]

\end{proof}

A Banach space $X$ is called and \emph{$L_1$-predual} if its dual $X^*$ is isometric to the space $L_1(\Omega, \Sigma,\lambda)$ for some measure space $(\Omega,\Sigma,\lambda)$.

Let $X$ be an $L_1$-predual. Then for each $t\in B_{X^*}$ and maximal probability measures $\mu,\nu\in M_t(B_{X^*})$, we have $\hom\mu=\hom\nu$ (see \cite[Chapter 7, \S 21, Theorem 7]{lacey} for the real case and \cite[Chapter 7, \S 23, Theorem 5]{lacey} for the complex case). We denote this uniquely defined measure as $D_t$.
Properties of the mapping $t\mapsto D_t$ are summarized in the following lemma.

\begin{lemma}
\label{l:decko-andproper}
Let $X$ be an $L_1$-predual. Then the following assertions hold:
\begin{enumerate}[(a)]
 \item For each $f\in C(B_{X^*})$, the function $Df(t)=\int_{B_{X^*}} f\di D_t$, $t\in B_{X^*}$, is Borel on $B_{X^*}$, homogeneous and strongly affine.
    \item The mapping $t\mapsto D_t$ is an affine homogeneous mapping of $B_{X^*}$ into $M_{\bnd,\ahom}(B_{X^*})$ such that, for each $t\in B_{X^*}$, $\int_{B_{X^*}} a\di D_t=a(t)$, $a\in X$, and  $\norm{D_t}\le 1$.
     \item For each $f\in C(B_{X^*})$, the function $DDf(t)=\int_{B_{X^*}} Df\di D_t$, $t\in B_{X^*}$, satisfies $DDf=Df$.
    \item Let $D\colon M(B_{X^*})\to M(B_{X^*})$ be defined as $D\mu(f)=\int_{B_{X^*}} Df\di\mu$, $f\in C(B_{X^*})$. Then $D$ is a linear mapping of $M(B_{X^*})$ into $M_{\ahom}(B_{X^*})$.
    \item For any $\mu\in M(B_{X^*})$, the measure $D\mu$ is the unique element of $M_{\bnd,\ahom}(B_{X^*})$ with $\int_{B_{X^*}} x\di\mu=\int_{B_{X^*}} x\di D\mu$ for each $x\in X$.
    \item The mapping $D$ is a projection of norm $1$ of $M(B_{X^*})$ onto $M_{\bnd,\ahom}(B_{X^*})$.
\end{enumerate}
\end{lemma}

\begin{proof}
(a) See \cite[Lemma 4.12]{petracek-spurny} and its proof. It is shown in \cite[Lemma 4.12]{petracek-spurny} that the function $Df$ is a uniform limit of linear combinations of upper semicontinuous functions on $B_{X^*}$, and thus $Df$ is Borel. 

(b) The mapping $t\mapsto D_t$ is affine and homogeneous by (a). Obviously, $D_t$ is antihomogeneous, $\norm{D_t}\le 1$ and $\int a \di D_t=a(t)$, $a\in X$,  for each $t\in B_{X^*}$. Since $D_t=\hom \mu$ for some maximal measure $\mu\in M_t(B_{X^*})$, $D_t$ is boundary by \cite[Chapter 7, \S 23, Lemma 10]{lacey}. 

(c) Let $t\in B_{X^*}$ and $f\in C(B_{X^*})$ be given. Using Lemma~\ref{l:approx-teleman} we find a bounded sequence $(a_n)\subset X$ such that $a_n\to Df$ $\abs{D_t}+\ep_t$-almost everywhere. Then 
\[
DDf(t)=\int Df\di D_t=\lim_{n\to\infty}\int a_n \di D_t=\lim_{n\to\infty} a_n(t)=Df(t).
\]

(d) The mapping $D$ is obviously linear. Further, 
\[
D(\hom f)(t)=D_t(\hom f)=\hom D_t(f)=D_t(f)=Df(t),\quad f\in C(B_{X^*}),
\]
and thus 
\[
(\hom D\mu)(f)=(D\mu)(\hom f)=\int D(\hom f)\di \mu=\int Df\di\mu=(D\mu)(f),\quad f\in C(B_{X^*}).
\]
Hence $D\mu=\hom D\mu$ is an antihomogeneous measure. 

(e) Let $\mu\in M(B_{X^*})$ be given. We write $\mu=a_1\mu_1-a_2\mu_2+i(a_3\mu_3-a_4\mu_4)$, where $a_1,\dots, a_4\ge 0$ and $\mu_1\dots,\mu_4\in M^1(B_{X^*})$. Let $t_i=r(\mu_i)$, $i=1,\dots,4$, be the barycenters. Then
\[
\int x\di\mu=a_1x(t_1)-a_2x(t_2)+i(a_3x(t_3)-a_4x(t_4)),\quad x\in X.
\]
We claim that $D\mu_i=D_{t_i}$, $i=1,\dots,4$.

Indeed, let $f\in C(B_{X^*})$ be given and let $(a_n)\subset X$ be a bounded sequence with $a_n\to Df$ $(\mu_i+\ep_{t_i})$-almost everywhere. Then
\[
D\mu_i(f)=\mu_i(Df)=\lim_{n\to\infty} \mu_i(a_n)=\lim_{n\to\infty} a_n(t_i)=Df(t_i)=D_{t_i}(f).
\]
Thus $D\mu_i=D_{t_i}$, $i=1,\dots, 4$.

Hence
\[
D\mu=a_1D_{t_1}-a_2D_{t_2}+i(a_3D_{t_3}-a_4D_{t_4})\in M_{\bnd,\ahom}(B_{X^*}).
\]
Further, 
\[
D\mu(x)=a_1x(t_1)-a_2x(t_2)+i(a_3x(t_3)-a_4x(t_4))=\mu(x).
\]
Hence $D\mu$ satisfies the required condition.

Assuming $\nu\in M_{\bnd,\ahom}(B_{X^*})$ is another measure with $\nu(x)=\mu(x)$, $x\in X$, then \cite[Theorem 2]{saabs-illinois} implies $\nu=D\mu$.

(f) The mapping $D$ is of norm at most $1$ and maps $M(B_{X^*})$ into the space $M_{\bnd,\ahom}(B_{X^*})$. If $\nu\in M_{\bnd,\ahom}(B_{X^*})$, then 
\[
D\nu(x)=\nu(Dx)=\nu(x),\quad x\in X,
\]
implies that $D\nu=\nu$ (here we use again \cite[Theorem 2]{saabs-illinois}). Hence $D$ is a projection of $M(B_{X^*})$ onto $M_{\bnd,\ahom}(B_{X^*})$.
\end{proof}

If $X$ and $E$ are Banach spaces, we denote by $\inxe$ the \emph{injective tensor product} of $X$ and $E$, see \cite[Chapter 3]{ryan-tensor}.

\begin{fact}
    \label{f:predual-injective}
    \begin{enumerate}[(a)]
    \item If $K$ is a simplex, the space $A(K)\wh{\otimes}_\ep E$, where $E$ is a Banach space, can be identified with $A(K,E)$ by the mapping $I\colon A(K)\wh{\otimes}_\ep E\to A(K,E)$ given by
    \[
    I(a\otimes x)(t)=a(t)x,\quad t\in K, a\in A(K), x\in E.
    \]
    \item   Let $X$ be an $L_1$-predual and $E$ be a Banach space. Then $X\wh{\otimes}_\ep E$ can be identified with the space $A_{\hom}(B_{X^*},E)$ via the isometric mapping $I\colon X\wh{\otimes}_\ep E\to A_{\hom}(B_{X^*},E)$ that satisfies
    \[
I(z\otimes x)(t)=t(z)x,\quad t\in B_{X^*}, z\in X, x\in E.
    \]
    \end{enumerate}
\end{fact}

\begin{proof}
 (a) The assertion follows from \cite[Section 3.2]{ryan-tensor} and \cite[Lemma 2.4]{lazar-sel}. Indeed, 
 if $u\in A(K)\otimes E$ is represented as $u=\sum_{i=1}^n a_i\otimes x_i$, where $a_1,\dots, a_n\in A(K)$ and $x_1,\dots, x_n\in E$, then the function $Iu(t)=\sum_{i=1}^n a_i(t)x_i$, $t\in K$, satisfies
 \[
 \norm{u}_{A(K)\wh{\otimes}_\ep E}=\sup_{s\in B_{A(K)^*}}\sup_{x^*\in B_{E^*}}\abs{\sum_{i=1}^n s(a_i)x^*(x_i)}
 \]
 and 
 \[
 \norm{Iu}_{A(K,E)}=\sup_{t\in K}\sup_{x^*\in B_{X^*}} \abs{\sum_{i=1}^n a_i(t)x^*(x_i)}
 \]
 Clearly, $\norm{Iu}_{A(K,E)}\le \norm{u}_{A(K)\wh{\otimes}_\ep E}$.

 On the other hand, for a fixed element $x^*\in B_{X^*}$, the function 
 \[
 s\in B_{A(K)^*}\mapsto \abs{\sum_{i=1}^n s(a_i)x^*(x_i)}
 \]
 is convex and continuous. Hence, it attains its supremum at some extreme point of $B_{A(K)^*}$. 
 Since $\ext B_{A(K)^*}\subset S_\ef\cdot K$, it follows that also $\norm{Iu}_{A(K,E)}\ge \norm{u}_{A(K)\wh{\otimes}_\ep E}$, and hence $I$ is an isometric mapping.

 It remains to show that the image of $I$ is dense in $A(K,E)$. To this end, let $f\in A(K,E)$ and $\ep>0$ be given. Then  $\Phi(t)=f(t)+\ep B_E$ is an affine lower semicontinuous multivalued mapping with nonempty closed convex values from a simplex $K$ to a Banach space $E$. 
 
 (A multivalued mapping $\Phi\colon L\to E$ on a  compact convex set $L$ is \emph{affine} if, for $t_1,t_2\in L$ and $\alpha\in (0,1)$ we have 
 \[
 \alpha\Phi(t_1)+(1-\alpha)\Phi(t_2)\subset \Phi(\alpha t_1+(1-\alpha)t_2).
 \]
It is \emph{lower semicontinuous} if, for any $U\subset E$ open, the set
\[
\Phi^{-1}(U)=\{t\in L\setsep \Phi(t)\cap U\neq\emptyset\}
\]
 is open in $L$.)
 
  Hence, by the mentioned \cite[Lemma 2.4]{lazar-sel}, $\Phi$ admits a selection $\phi=\sum_{i=1}^n f_i\otimes x_i\in A(K)\otimes E$, i.e, $\phi(t)$ is an element of $\Phi(t)$. Thus, any element of $A(K,E)$ can be uniformly approximated by elements of $A(K)\otimes E$. 

    (b) The assertion follows from \cite[Lemma 2.2, p.\,173]{lali} (the real case) and \cite[the proof of Theorem 4.2]{olsen-sel} (the complex case) as in the previous case.     
\end{proof}

It follows from \cite[Theorem 2]{saabs-illinois}, that, for an $L_1$-predual $X$ and a Banach space $E$, the dual $(\inxe)^*$ of $\inxe$ can be identified with $M_{\bnd,\ahom}(B_{X^*},E^*)$. We mention the following fact.

\begin{fact}
    \label{f:weakly}
    The space $M_{\bnd,\ahom}(B_{X^*},E^*)$ is weakly closed in $M(B_{X^*},E^*)$.
\end{fact}

\begin{proof}
First we observe that the space $ M_{\ahom}(B_{X^*},E^*)$ is even weak$^*$ closed.
Thus it is enough to check that $M_{\bnd}(B_{X^*},E^*)$ is weakly closed. If $(\mu_i)_{i\in I}$ is a net of boundary measure in $M_{\bnd}(B_{X^*},E^*)$ converging weakly to $\mu\in M(B_{X^*},E^*)$, for an element $x\in E$ we have $(\mu_i)_x\to \mu_x$ weakly. If $g\in C(B_{X^*},\er)$ is a given function, $B=\{t\in B_{X^*}\setsep g^*(t)=g(t)\}$ and $A\subset B_{X^*}\setminus B$ is Borel, then $(\mu_i)_x(A)=0$. By weak convergence, $\mu_x(A)=0$, which yields that $\mu_x$ is carried by $B$. Hence $\mu_x\in M_{\bnd}(B_{X^*})$. Since $x\in E$ is arbitrary, $\mu\in M_{\bnd}(B_{X^*},E^*)$.
\end{proof}

\section{The representing measure for operators on $X\wh{\otimes}_\ep E$}
\label{s:sboake}

We recall a construction of the representing measure for operators on $A(K,E)$, where $K$ is a simplex, and $\inxe$, where $X$ is an $L_1$-predual.

\begin{definition}\label{d:repremeasure}
Let $K$ be a simplex and $E,F$ be Banach spaces. If $T\colon A(K,E)\to F$ is an operator, its \emph{representing measure} is a finitely additive mapping $G$ defined on $\Borel(K)$ with values in $L(E,F^{**})$ such that
\begin{itemize}
\item for each $y^*\in F^*$, the measure $G_{y^*}\colon \Borel(K)\to E^*$ 
is a regular Borel countably additive boundary $E^*$-valued vector measure on $K$, i.e., $G_{y^*}\in M_{\bnd}(K,E^*)$;
\item the semivariation of $G$ 
is finite, i.e., $\norm{G}(K)<\infty$;
\item the mapping $y^*\to G_{y^*}|_{A(K,E)}$ is weak$^*$-to-weak$^*$ continuous from $F^*$ to $A(K,E)^*$;
\item $T^*y^*=G_{y^*}$ on $A(K,E)$ for all $y^*\in F^*$; and
\item $\norm{T}=\norm{G}(K).$
\end{itemize}
\end{definition}

The following result is from \cite{saab-saab}.

\begin{thm}\label{t:saab-saab}
Let $K$ be a simplex and $E,F$ be Banach spaces.
Let $T\colon A(K,E)\to F$ be an operator, then $T$ has a unique representing measure $G$.
\end{thm}

\begin{proof}
We recall the construction of $G$ from \cite[Theorem 2]{saab-saab}. 

Let $B\in \Borel(K)$ and $x\in E$ be given, then $\varphi_{B,x}\in M(K,E^*)^*$ is defined as $\varphi_{B,x}(\mu)=\mu(B)x$, $\mu\in M(K,E^*)$. By \cite[Theorem 3.4]{saab-aeq} there exists a linear mapping $S\colon (A(K,E))^*\to M_{\bnd}(K,E^*)$ such that $\norm{S(l)}=\norm{l}$ for every $l\in (A(K,E))^*$ and $S(l)=l$ on $A(K,E)$.
The required representing measure $G$ is then defined as
\[
G(B)x=T^{**}\circ S^*(\varphi_{B,x}),\quad B\in\Borel(K), x\in E.
\]
Then \cite[Theorem 2]{saab-saab} proves that this formula works and that $G$ is uniquely determined by the properties in Definition~\ref{d:repremeasure}.
\end{proof}

\begin{remark}
It is remarked in \cite[Remark, p.\,395]{saab-saab} that if $T\colon A(K,E)\to F$ has a representing measure $G$, it admits an extension $S\colon C(K,E)\to F^{**}$ preserving the norm.
    It is proved in \cite[Theorem 3]{saab-saab} that, if every operator on $A(K,E)$ for a compact convex set $K$ admits an extension preserving the norm, then $K$ is a simplex. From thence it follows the restriction of our study to simplices.

    Similarly, \cite[Proposition 2]{saabs-rocky-crucial} explains why the possibility to extend any operator $T\colon \inxe\to F$ to an operator $\wh{T}\colon C(B_{X^*},E)\to F^{**}$ with preservation of norm gives that $X$ is an $L_1$-predual.
\end{remark}

A similar representation is available for $L_1$-preduals.

\begin{definition}\label{d:repremeasure-predual}
Let $X$ be an $L_1$-predual and $E,F$ be Banach spaces. If $T\colon \inxe\to F$ is an operator, its \emph{representing measure} is a finitely additive mapping $G$ defined on $\Borel(B_{X^*})$ with values in $L(E,F^{**})$ such that
\begin{itemize}
\item for each $y^*\in F^*$, the measure $G_{y^*}\colon \Borel(B_{X^*})\to E^*$ 
is a regular Borel countably additive boundary antihomogeneous $E^*$-valued vector measure on $B_{X^*}$, i.e., $G_{y^*}\in M_{\bnd,\ahom}(B_{X^*},E^*)$;
\item the semivariation of $G$ 
is finite, i.e., $\norm{G}(B_{X^*})<\infty$;
\item the mapping $y^*\to G_{y^*}|_{\inxe}$ is weak$^*$-to-weak$^*$ continuous from $F^*$ to $(X\wh{\otimes}_\ep E)^*$;
\item $T^*y^*=G_{y^*}$ on $X\wh{\otimes}_\ep E$ for all $y^*\in F^*$; and
\item $\norm{T}=\norm{G}(B_{X^*}).$
\end{itemize}
\end{definition}

The following result is from \cite{saabs-rocky-crucial}.

\begin{thm}\label{t:saab-saab-predual}
Let $X$ be an $L_1$-predual and $E,F$ be Banach spaces.
Let $T\colon \inxe\to F$ be an operator, then $T$ has a unique representing measure $G$.
\end{thm}

\begin{proof}
We recall the construction of $G$ from \cite[Remark 1]{saabs-rocky-crucial}. 

Let $B\in \Borel(B_{X^*})$ and $x\in E$ be given, then $\varphi_{B,x}\in M(B_{X^*},E^*)^*$ is defined as $\varphi_{B,x}(\mu)=\mu(B)x$, $\mu\in M(B_{X^*},E^*)$. By \cite[Theorem 2]{saabs-illinois} there exists a linear mapping $S\colon \inxe^*\to M_{\bnd,\ahom}(B_{X^*},E^*)$ such that $\norm{S(l)}=\norm{l}$ for every $l\in \inxe^*$ and $S(l)=l$ on $\inxe$.
The required representing measure $G$ is then defined as
\[
G(B)x=T^{**}\circ S^*(\varphi_{B,x}),\quad B\in\Borel(B_{X^*}), x\in E.
\]
Then \cite[Remark 1]{saabs-rocky-crucial} shows that the formula works. 

Concerning the uniqueness, let $G'\colon \Borel(B_{X^*})\to L(E,F^{**})$ be another representing measure with the properties described in Definition~\ref{d:repremeasure-predual}.
Then for each $y^*\in F^*$, $z\in X$ and $x\in E$, we have
\[
\int_{B_{X^*}}z\otimes x\di G_{y^*}=T^*y^*(z\otimes x)=\int_{B_{X^*}}z\otimes x\di (G')_{y^*}.
\]
Hence 
\[
G_{y^*}-(G')_{y^*}\in M_{\bnd,\ahom}(B_{X^*},E^*)\cap (\inxe)^\perp.
\]
Since the latter set is equal to zero by \cite[Theorem 2]{saabs-illinois}, $G_{y^*}=(G')_{y^*}$. Hence $G=G'$.
\end{proof}

\begin{prop}
    \label{p:vztaha-tensor}
    Let $K$ be a simplex and $X=A(K)$. Let $E,F$ be Banach spaces and $I\colon A(K)\wh{\otimes}_\ep E\to A(K,E)$ be the identification from Fact~\ref{f:predual-injective}. Let $T\colon A(K,E)\to F$ be an operator with the representing measure $G\colon\Borel(K)\to L(E,F^{**})$ and let the operator $T'\colon A(K)\wh{\otimes}_\ep E\to F$ satisfying $T'=T\circ I$ have the representing measure $G'\colon \Borel(B_{X^*})\to L(E,F^{**})$. 
    
    Then, for each $y^*\in F^*$,
    \[
    (G')_{y^*}=\vhom \phi(G_{y^*})\quad \text{and}\quad  G_{y^*}=H(\theta^{-1}(G')_{y^*})
    \]
    (here $H$ is the mapping from Lemma~\ref{l:prenos}).
\end{prop}

\begin{proof}
Given $y^*\in F^*$, let $f\in A(K)$ and $x\in E$ be given. Then
\[
\begin{aligned}
\int_{B_{X^*}}&(f\otimes x)\di (G')_{y^*}=((T')^*y^*)(f\otimes x)=y^*(T'(f\otimes x))=y^*(T(I(f\otimes x)))\\
&=(T^*y^*)(I(f\otimes x))=\int_K f(t)x\di G_{y^*}(t)=\int_{B_{X^*}} (f\otimes x)\di \phi(G_{y^*})\\
&=\int_{B_{X^*}} (f\otimes x) \di \hom \phi(G_{y^*}).
\end{aligned}
\]
Hence $\vhom\phi(G_{y^*})-(G')_{y^*}$ is a boundary antihomogeneous measure on $B_{X^*}$ (see Lemma~\ref{l:prenos}) that annihilates $A_{\hom}(B_{X^*},E)=\inxe$. Thus, it is equal to zero by \cite[Theorem 2]{saabs-illinois}.
In other words, $(G')_{y^*}=\vhom\phi(G_{y^*})$.

Concerning the other equality, we observe that for $f\in A(K)$ and $x\in E$ we have
\[
\begin{aligned}
\int_{K}&(f\otimes x)\di H(\theta^{-1}(G')_{y^*})=\int_{(\alpha,t)\in \TT\times K}\alpha(f\otimes x)(t)\di \theta^{-1}((G')_{y^*})\\
&=\int_{B_{X^*}}s(f)x\di (G')_{y^*}(s)
=\int_{B_{X^*}} (f\otimes x)\di (G')_{y^*}=((T')^*y^*)(f\otimes x)\\
&=y^*(T'(f\otimes x))=y^*(T(I(f\otimes x)))=(T^*y^*)(I(f\otimes x))=\int_K f(t)x\di G_{y^*}(t).    
\end{aligned}
\]
Hence $G_{y^*}-H(\theta^{-1}(G')_{y^*})\in M_{\bnd}(K,E^*)\cap A(K,E)^\perp$ (see Lemma~\ref{l:prenos}(c)), and thus it is equal to zero by \cite[Theorem 3.4]{saab-aeq}. Thus $G_{y^*}=H(\theta^{-1}(G')_{y^*})$.
\end{proof}

\section{Unconditionally convergent operators on $\inxe$}
\label{s:uncoopare}

Here we recall \cite[Proposition 1]{saabs-rocky-crucial}, which shows that any unconditionally converging operator on $\inxe$ (where $X$ is an $L_1$-predual) is strongly bounded.

\begin{thm}
\label{t:uncond-sbdd}
Let $X$ be an $L_1$-predual and $E,F$ be Banach spaces. Then any unconditionally converging operator $T\colon \inxe \to F$ is strongly bounded.
\end{thm}

\begin{proof}
The result is stated as \cite[Proposition 1]{saabs-rocky-crucial}, which was shown during the proof of \cite[Theorem 3]{saabs-illinois}. However, a key ingredient was the result in \cite{johnson-zippin} that an $L_1$-predual $X$ has  the property (V). Unfortunately, the result in \cite{johnson-zippin} is  formulated only for real Banach spaces. For complex spaces it is possible to repair the proof as follows: Let $T\colon X\to Y$ be an unconditionally converging operator on a complex $L_1$-predual. Let $(x_n)\subset B_X$ be a given sequence. Then there exists a separable $L_1$-predual $Y$ satisfying $(x_n)\subset Y\subset X$ (see \cite[Chapter 7, \S 23, Lemma 1]{lacey}). By \cite[Theorem]{lusky-compl}, $Y$ is $1$-complemented in  a separable $C^*$-algebra $A$ by a projection $P$. Since $A$ has  the property (V) (see  \cite[Corollary 6]{pfitzner-mathann}), the operator $T|_{Y}\circ P$ is weakly compact. Hence $T|_Y $ is weakly compact, which gives a subsequence $(x_{n_k})$ such that $T(x_{n_k})$ converges weakly. Hence $T$ is weakly compact and thus $X$ has the property (V).
    
\end{proof}
\section{Strongly bounded operators on $\inxe$}
\label{s:strongly}

Let now $T\colon \inxe\to F$ be a strongly bounded operator.
Now we can obtain an analog of Theorem~\ref{t:sbo}. Before embarking on its proof, we need a lemma.

\begin{lemma}
    \label{l:odhad}
    Let $G\colon \Borel(K)\to L(E,F^{**})$ be a bounded vector measure and $y^*\in F^*$, $x\in B_E$. Then
    \[
 \abs{G_{y^*,x}}\le\abs{G_{y^*}}.
    \]
\end{lemma}

\begin{proof}
For each $E\in\Borel(K)$ we have
\[
\begin{aligned}
\abs{G_{y^*,x}}(E)&=\sup_{(E_j)\in \pi(E)}\sum \abs{G_{y^*,x}(E_j)}=\sup_{(E_j)\in \pi(E)}\sum \abs{(G(E_j)x)(y^*)}
\end{aligned}
\]
and
\[
\begin{aligned}
\abs{G_{y^*}}(E)=\sup_{(E_j)\in \pi(E)}\sum \norm{G_{y^*}(E_j)}.
\end{aligned}
\]
For $\ep>0$, let $(E_j)\in\pi(E)$ be such that
\[
\abs{G_{y^*,x}}(E)-\ep<\sum \abs{(G(E_j)x)(y^*)}.
\]
Since 
\[
\norm{G_{y^*}(E_j)}\ge \abs{(G_{y^*}(E_j))(x)}=\abs{(G(E_j)x)(y^*)},
\]
we get
\[
\abs{G_{y^*,x}}(E)-\ep<\sum \abs{(G(E_j)x)(y^*)}\le \sum \norm{G_{y^*}(E_j)}\le \abs{G_{y^*}}(E).
\]
This concludes the proof.
\end{proof}

\begin{thm}
\label{t:controlmeasure}
Let $X$ be an $L_1$-predual and $E,F$ be Banach spaces.
Let $T\colon \inxe\to F$ be a strongly bounded operator with the representing measure $G$. 
\begin{enumerate}[(a)]
\item Then $G$ has values in $L(E,F)$.
\item There exists a maximal $\TT$-invariant measure $\lambda\in M^+(B_{X^*})$ such that 
\[
\lim_{B\in \Borel(K),\lambda(B)\to 0}\norm{G}(B)=0
\]
(this measure is called a \emph{control measure of $G$.})
\end{enumerate}
\end{thm}

\begin{proof}
(a) Let $B\in\Borel(B_{X^*})$ be given. For $x\in B_E$ we want to prove that the element $y^{**}=G(B)x\in F^{**}$ is in fact in $F$.
To this end, it is enough to check that the restriction of $y^{**}$ to $B_{F^*}$ is weak$^*$ continuous. 

If we assume the contrary, we can find $\eta>0$ and a net $(y_i^*)_{i\in I}\subset B_{F^*}$ converging weak$^*$ to zero such that $\abs{y^{**}(y_i^*)}>\eta$, $i\in I$. Since $G$ is strongly bounded, the set
$\{\abs{G_{y_i^*}}\setsep i\in I\}$ is a bounded relatively weakly compact set in $M(B_{X^*})$. 
Since $\abs{G_{y_i^*,x}}\le \abs{G_{y_i^*}}$ (see Lemma~\ref{l:odhad}), the set
$\{\abs{G_{y_i^*,x}}\setsep i\in I\}$ is also relatively weakly compact.
Hence also the set $\{G_{y_i^*,x}\setsep i\in I\}$ is relatively weakly compact in $M(B_{X^*})$, see \cite[Chapter 1, Corollary 4]{diesteluhl} and \cite[Theorem IV.9.2, p.\,306]{dunford-schwartz}.
Thus, we may assume that $(G_{y_i^*,x})_{i\in I}$ converges weakly to $\mu\in M(B_{X^*})$. Since the measures $G_{y_i^*,x}$ are boundary and homogeneous, and $M_{\bnd,\ahom}(B_{X^*})$ is weakly closed (see Fact~\ref{f:weakly}), the measure $\mu$ is also in $M_{\bnd,\ahom}(B_{X^*})$.
Since the mapping $y^*\mapsto G_{y^*}|_{\inxe}$ is weak$^*$ continuous, for each $z\in X$
we know that the mapping $y^*\mapsto G_{y^*}(z\otimes x)$ is continuous on $B_{F^*}$. Thus, we obtain
\[
0=\lim_{i\in I} \int_{B_{X^*}} (z\otimes x)\di G_{y_i^*}=\lim_{i\in I}\int_{B_{X^*}} z\di G_{y_i^*,x}=\int_{B_{X^*}} z\di \mu,\quad z\in X.
\]
Hence $\mu\in M_{\bnd,\ahom}(B_{X^*})\cap X^\perp$, which gives $\mu=0$ by \cite[Theorem 2]{saabs-illinois}.

On the other hand,
\[
\eta\le \lim_{i\in I}\abs{y^{**}(y_i^*)}=\lim_{i\in I}\abs{(G(B)x)(y_i^*)}=\lim_{i\in I}\abs{G_{y_i^*,x}(B)}=\abs{\mu(B)}=0.
\]
This contradiction proves that $G(B)x\in F$.

(b) Since $G$ is strongly bounded, the set $H=\{\abs{G_{y^*}}\setsep y^*\in B_{F^*}\}$ is relatively weakly compact. By \cite[Theorem IV.9.2, p.\,306]{dunford-schwartz}, there exists a positive measure $\lambda$ on the $\sigma$-algebra $\Borel(B_{X^*})$ such that  
\begin{itemize}
\item $\lim_{B\in\Borel(B_{X^*}),\lambda(B)\to 0}\norm{G}(B)=0$;
\item $\lambda\in M^+(B_{X^*})$;
\item $\lambda$ is maximal and $\TT$-invariant.    
\end{itemize}
Indeed, the first property is stated in the cited theorem, the second and the third property follows from the proof of the theorem, since $\lambda$ is defined as an absolutely convergent sum of measures from $H$ which are maximal and $\TT$-invariant (see \cite[p.\,307]{dunford-schwartz}).
Hence $\lambda$ is the desired measure.
\end{proof}

\begin{lemma}
    \label{l:approx}
 Let $G$ be a representing measure for $T\colon\inxe\to F$ that is strongly bounded. Then for each $B\in\Borel(B_{X^*})$ and $\ep>0$ there exists $z\in  X$ such that the operator $S\in L(E,F)$ defined as $Se=T(z\otimes e)$, $e\in E$, satisfies $\norm{G(B)-S}<\ep$.
\end{lemma}

\begin{proof}
We know that $G$  it is strongly bounded with a maximal control $\TT$-invariant measure $\lambda\in M^+(B_{X^*})$. Let $B\in\Borel(B_{X^*})$ be arbitrary and $\ep\in (0,1)$ be fixed.
Let $\delta\in (0,\ep)$ be chosen such that $\norm{G}(E)<\ep$ whenever $\lambda(E)<\delta$, $E\in\Borel(B_{X^*})$. We find a compact set $L\subset B_{X^*}$ and an open set $U\subset B_{X^*}$ with $L\subset B\subset U$ such that $\lambda(U\setminus B)<\delta$. Let $f\in C(B_{X^*})$ be such that $\chi_L\le f\le \chi_U$.

Let $Df(t)=D_t(f)$, $t\in B_{X^*}$, where $D_t$ is the unique measure defined in Section~\ref{ss:preduals}. Then $Df$ is a strongly affine function (see Lemma~\ref{l:decko-andproper}) and hence there exists a sequence $(a_n)\subset X$ with $\norm{a_n}\le \norm{Df}\le \norm{f}\le 1$ converging to $Df$ $\lambda$-almost everywhere (see Lemma~\ref{l:approx-teleman}). Then $\int_{B_{X^*}} \abs{a_n-Df}\di\lambda\to 0$. Let $z$ be an element of the sequence $(a_n)$ satisfying $\int_{B_{X^*}} \abs{z-Df}\di\lambda<\delta^2$. We define the operator $S\colon E\to F$ as $Se=T(z\otimes e)$, $e\in E$. 

We claim $\norm{G(B)-S}<\ep(4+\norm{T})$. Indeed, let $x\in B_E$ and $y^*\in B_{F^*}$ be arbitrary.
Then
\begin{equation}
\label{eq:odhadyG}
\begin{gathered}
 \abs{y^*(G(B)x-Sx)}=\abs{G_{y^*,x}(B)-y^*(\int_{B_{X^*}} (z\otimes x)\di G)}\\
 \le \abs{\int_{B_{X^*}}(\chi_B-f)\di G_{y^*,x}}+\abs{\int_{B_{X^*}} f\di G_{y^*,x}-\int_{B_{X^*}} Df\di G_{y^*,x}}+\abs{\int_{B_{X^*}} (Df-z)\di G_{y^*,x}}.
\end{gathered}
\end{equation}
Let $H=\{s\in B_{X^*}\setsep \abs{z(s)-Df(s)}\ge \delta\}$. Then
\[
\lambda(H)=\int_H 1\di\lambda=\frac{1}{\delta}\int_H\delta\di\lambda\le 
\frac1{\delta}\int_{B_{X^*}} \abs{z-Df}\di\lambda\le \frac{\delta^2}{\delta}=\delta.
\]
Thus $\norm{G}(H)<\ep$. Hence
\[
\begin{aligned}
\abs{\int_{B_{X^*}} Df\di G_{y^*,x}-\int_{B_{X^*}} z\di G_{y^*,x}}&\le
\int_H \abs{Df-z}\di \abs{G_{y^*,x}}+\int_{B_{X^*}\setminus H} \abs{Df-z}\di\abs{G_{y^*,x}}\\
&\le 2\norm{f}\abs{G_{y^*,x}}(H)+\delta\abs{G_{y^*,x}}(K)\\
&\le 2\abs{G_{y^*}}(H)+\ep\norm{T}\le 2\norm{G}(H)+\ep\norm{T}\\
&\le \ep(2+\norm{T}).
\end{aligned}
\]
Since $G_{y^*,x}$ is boundary and antihomogeneous, $D(G_{y^*,x})=G_{y^*,x}$ (see Lemma~\ref{l:decko-andproper}(f)), and thus
\[
\int_{B_{X^*}} f\di G_{y^*,x}=\int_{B_{X^*}} f\di DG_{y^*,x}=\int_{B_{X^*}} Df\di G_{y^*,x}.
\]
Hence, we obtain from \eqref{eq:odhadyG}
\[
\begin{aligned}
\abs{y^*(G(B)x-Sx)}&\le 2\abs{G_{y^*,x}}(U\setminus L) +\ep(2+\norm{T})\\
&\le 2\abs{G_{y^*}}(U\setminus L)+\ep (2+\norm{T})\\
&\le 2\norm{G}(U\setminus L)+\ep(2+\norm{T})\\
&\le \ep(2+2+\norm{T}).
\end{aligned}
\]
Since $y^*\in B_{F^*}$ and $x\in B_{E}$ are arbitrary,  we get $\norm{G(B)-S}\le \ep(4+\norm{T})$, which concludes the proof.
\end{proof}

\begin{lemma}
\label{l:l1preduincreas}
Let $X$ be a Banach space and $Y_1\subset Y_2\subset Y_3\subset \cdots\subset X$ be a nondecreasing sequence of closed subspaces that are $L_1$-preduals. Then $Y=\ov{\bigcup_{n=1}^\infty Y_n}$ is an $L_1$-predual.
\end{lemma}

\begin{proof}
We verify condition (d) of \cite[Chapter 7, \S 23, Theorem 2]{lacey}. So let $A\subset Y$ be a finite set and $\ep>0$ be given.  Let $n\in\en$ be such that $\dist(x,Y_n)<\ep$, $x\in A$. For each $x\in A$ we select $y_x\in Y_n$ such that $\norm{x-y_x}<\ep$. Then $C=\{y_x\setsep x\in A\}\subset Y_n$ is a finite set in an $L_1$-predual $Y_n$, and thus there exists by \cite[Chapter 7, \S 23, Theorem 2(d)]{lacey} an operator $J\colon \ell_\infty(n,\ef)\to Y_n$ such that $(1+\ep)^{-1}\norm{y}\le\norm{Jy}\le (1+\ep)\norm{y}$, $y\in\ell_\infty(n,\ef)$ and $\dist(y_x, \Rng J)<\ep$, $x\in A$. Then $\dist(x,\Rng J)<2\ep$. Hence $Y$ satisfies the required condition, which concludes the proof.    
\end{proof}

\begin{lemma}
\label{l:separablredukce}
    Let $X$ be an $L_1$-predual and $E$ be a Banach space. Let $\F\subset C(B_{X^*},E)$ be separable and $\M\subset M_{\bnd}(B_{X^*})$ countable. 

    Then there exists a separable $L_1$-predual $Y\subset X$ such that for the restriction mapping $\pi\colon B_{X^*}\to B_{Y^*}$ the following assertions hold:
    \begin{itemize}
\item[(a)] For each $f\in \F$ there exists a function $f'\in C(B_{Y^*},E)$ with $f=f'\circ \pi$.
\item[(b)] The measure $\pi(\mu)$ is boundary for each $\mu\in \M$.        
    \end{itemize}
\end{lemma}

\begin{proof}
For a space $Z\subset X$ we denote by $\pi_Z$ the restriction mapping $\pi_Z\colon X^*\to Z^*$.
Since any measure $\mu\in M_{\bnd}(B_{X^*})$ can be written as $\mu=a_1\mu_1-a_2\mu_2+i(a_3\mu_3-a4\mu_4)$ with $a_1,\dots, a_4$ positive and $\mu_1,\dots,\mu_4\in M^1(B_{X^*})$ maximal, we may assume that $\M$ consists of positive measures .

    Using \cite[Lemma 6.2]{Kaspu-studia} we find a separable subspace $A\subset X$ such that for each $f\in\F$ there exists $f'\in C(B_{A^*},E)$ with $f=f'\circ\pi_A$.
    By \cite[Chapter 7, \S 23, Lemma 1]{lacey}, we can find a separable $L_1$-predual $A\subset Y_0\subset X$. Then the functions from $\F$ factorize through $\pi_{Y_0}$.

For $\A\subset C(B_{X^*})$, let $\W(\A)$ denote the following set of real functions on $B_{X^*}$:
\[
\W(\A)=\{(\Re a_1+c_1)\vee\cdots\vee (\Re a_m+c_m)\setsep m\in\en, a_1,\dots,a_m\in \A, c_1,\dots, c_m\in\qe\}.
\]
Now we construct inductively countable sets $\A_0\subset \A_1\subset \cdots\subset X$ such that:
\begin{enumerate}[(1)]
\item $\A_0$ is a dense subset of $Y_0$;
\item $\ov{\A_n}$ is an $L_1$-predual in $X$ for every $n\ge 0$;
\item for each $n\ge 0$ the following holds:
\[
\forall \mu\in\M\forall \ep>0\forall g\in \W(\A_n)\exists h\in -\W(\A_{n+1})\colon g\le h \text{ and }\mu(h-g)<\ep.
\]    
\end{enumerate}

The construction runs as follows. In the first step, we select a countable family $\B\subset X$ such that (3) holds for $\A_0$ and $-\W(\B)$. This is possible, because any $\mu\in\M$ satisfies by maximality 
\[
\mu(g)=\mu(g^*)=\inf\{\mu(k)\setsep k\ge g, k\text{ concave}\}=\inf\{\mu(k)\setsep k\ge g, k\in -\W(A(B_{X^*},\er)\}
\]
for any $g\in C(B_{X^*},\er)$ convex. Then we use \cite[Chapter 7, \S 23, Lemma 1]{lacey} to obtain a separable $L_1$-predual $B$ with  $\ov{\span}\B\subset B\subset X$. Let $\A_1$ be a countable dense set in $B$ which contains $\B\cup \A_0$.

In the inductive step, we pick a countable family $\B\subset X$ such that (3) is satisfied for $\W(\A_n)$ by functions from $-\W(\B)$. Again we enlarge the set $\B\cup\A_n$ into a countable set $\A_{n+1}$ such that $\ov{\A_{n+1}}$ is an $L_1$-predual.
This finishes the construction.

We take $Y=\ov{\bigcup_{n=0}^\infty \A_n}$.  By Lemma~\ref{l:l1preduincreas}, $Y$ is an $L_1$-predual. Let $\pi\colon B_{X^*}\to B_{Y^*}$ be the restriction mapping.
Let $g\in C(B_{Y^*},\er)$ be a convex function and $\ep>0$. Then there exists a convex function $g'\in \W(A(B_{Y^*},\er))$ such that $\norm{g-g'}_{C(B_{Y^*})}<\ep$, see \cite[Corollary I.1.3]{alfsen}.
Then $g'=(\Re y_1+c_1)\vee\cdots\vee (\Re y_m+c_m)$, where $y_1,\dots, y_m\in Y$, $c_1,\dots, c_m\in\er$ and $m\in\en$. By the density of $\bigcup_{n=0}^\infty \A_n$ in $Y$, we can find a function $g''=(\Re z_1+d_1)\vee\cdots\vee (\Re z_m+d_m)$ with $\norm{g'-g''}_{C(B_{Y^*})}<\ep$, where $z_1,\dots, z_m\in \bigcup_{n=0}^\infty \A_n$ and $d_1,\dots, d_m\in\qe$. Then there exists $n\in\en$ such that $z_1,\dots, z_m\in \A_{n}$. Then the function
\[
g'''(t)=(\Re z_1(t)+d_1)\vee\cdots\vee (\Re z_m(t)+d_m),\quad t\in B_{X^*},
\]
is in $\W(\A_n)$ and satisfy $g'''=g''\circ \pi$.
Let $h\in-\W(\A_{n+1})$ satisfy $\mu(h-g''')<\ep$ and $h\ge g'''$ on $B_{X^*}$. Then $h=(\Re x_1+e_1)\wedge\cdots\wedge (\Re x_k+e_k)$ for some $x_1,\dots, x_k\in Y$, $e_1,\dots, e_k\in \qe$ and $k\in\en$. Then 
\[
h'(s)=(\Re x_1(s)+e_1)\wedge\cdots\wedge (\Re x_k(s)+e_k),\quad s\in B_{Y^*},
\]
is a concave continuous function on $B_{Y^*}$ with $h=h'\circ \pi$. Set $h''=h'+2\ep$. 
Then 
\[
g\le g'+\ep\le g''+2\ep\le h'+2\ep=h''
\]
and
\[
\begin{aligned}
(\pi\mu)(g^*)&\le (\pi\mu)(h'')=\mu(h''\circ\pi)=\mu(h'\circ\pi)+2\ep\norm{\mu}\\
&=\mu(h)+2\ep\norm{\mu}\le \mu(g''')+\ep(1+2\norm{\mu})\\
&=\mu(g''\circ\pi)+\ep(1+2\norm{\mu})=(\pi\mu)(g'')+\ep(1+2\norm{\mu})\\
&\le (\pi\mu)(g+2\ep)+\ep(1+2\norm{\mu})\le (\pi\mu)(g)+\ep(1+4\norm{\mu}).
\end{aligned}
\]
Thus $(\pi\mu)(g^*)=(\pi\mu)(g)$, which gives the maximality of $\pi\mu$.
\end{proof}

\begin{lemma}
    \label{l:redukce}
    Let $X$ be an $L_1$-predual and $E,F$ be Banach spaces.
Let $G$ be the representing measure for a strongly bounded operator $T\colon \inxe\to F$ with a maximal invariant control measure $\lambda$.    
Let $Y\subset X$ be an $L_1$-predual and $\pi\colon B_{X^*}\to B_{Y^*}$ be the restriction mapping satisfying that $\pi\lambda$ is maximal. Let $T'\colon Y\wh{\otimes}_\ep E=A_{\hom}(B_{Y^*},E)\to F$ be defined as $T'f'=T(f'\circ \pi)$, $f'\in A_{\hom}(B_{Y^*},E)$.

Then $T'$ has the representing measure $G'=\pi(G)$, where $G'(B')=G(\pi^{-1}(B'))$, $B'\in \Borel(B_{Y^*})$. Moreover, $\pi\lambda$ is a maximal invariant control measure for $G'$. 
\end{lemma}

\begin{proof}
Let $G'\colon \Borel(B_{Y^*})\to L(E,F^{**})$ be the representing measure for $T'$.
For each $f'\in Y\wh{\otimes}_\ep E$ and $y^*\in B_{F^*}$ we have
\[
y^*(T(f'\circ\pi))=(T^*y^*)(f'\circ\pi)=\int_{B_{X^*}} (f'\circ\pi)\di G_{y^*}=\int_{B_{Y^*}} f'\di \pi(G_{y^*})
\]
and
\[
y^*(T'f')=\int_{B_{Y^*}} f'\di (G')_{y^*}.
\]
Since $T'f'=T(f'\circ\pi)$, we have $(G')_{y^*}-\pi(G_{y^*})\in A_{\hom}(B_{Y^*},E)^\perp$.

Next, we claim that $\pi(G_y^*)$ is a boundary measure. Indeed, we have
\begin{equation}
\label{eq:control}
\abs{\pi(G_{y^*})}\le \pi(\abs{G_{y^*}}).
\end{equation}
Let $g\in C(B_{Y^*},\er)$ be a convex function and $H\subset \{s\in B_{Y^*}\setsep g^*(s)>g(s)\}$ be any Borel set. 
Since $\pi\lambda$ is maximal, $0=\pi\lambda(H)=\lambda(\pi^{-1}(H))$. Hence $\abs{G_{y^*}}(\pi^{-1}(H))=0$, which in turn yields $\pi(\abs{G_{y^*}})(H)=0$. In view of \eqref{eq:control}, $\pi(G_{y^*})$ is zero on $B$. Since $g$ is arbitrary, the measure $\pi(G_{y^*})$ is boundary.

Since $Y$ is an $L_1$-predual, $(A_{\hom}(B_{Y^*},E))^\perp\cap M_{\bnd,\ahom}(B_{Y^*},E^*)=\{0\}$, and thus $(G')_{y^*}=\pi(G_{y^*})$.
Thus $\pi(G)=G'$ is  the representing measure for $T'$. Moreover, we have
\[
\lim_{B\in\Borel(B_{Y^*}),\pi\lambda(B)\to 0}\norm{G'}(B)=0.
\]
Indeed, let $\ep>0$ be given and $\delta>0$ is chosen such that $\norm{G}(B)<\ep$ whenever $\lambda(B)<\ep$, $B\in\Borel (B_{X^*})$. We fix a Borel set $B\subset B_{Y^*}$ with $\pi\lambda(B)<\delta$. Then
$\lambda(\pi^{-1}(B))<\delta$, and thus $\norm{G}(\pi^{-1}(B))<\ep$. Hence for any $y^*\in B_{F^*}$ we have
\[
\abs{(G')_{y^*}}(B)=\abs{\pi(G_{y^*})}(B)\le \pi(\abs{G_{y^*}})(B)=\abs{G_{y^*}}(\pi^{-1}(B))\le \norm{G}(\pi^{-1}(B))\le \ep.
\]
Hence $\norm{G'}(B)\le \ep$.
\end{proof}

\begin{cor}
    \label{c:operbecka}
    Let $X$ be an $L_1$-predual and $E,F$ be Banach spaces.
Let $G$ be the representing measure for a strongly bounded operator $T\colon \inxe\to F$. Let $\I$ denote a category of operators between Banach spaces which is closed with respect to the norm convergence. 

Then $G(B)\in\I$ for each $B\in \Borel(B_{X^*})$ provided $T\in\I$. This in particular applies to the categories of operators listed in Section~\ref{sec:oper}.    
\end{cor}

\begin{proof}
The assertion immediately follows from Lemma~\ref{l:approx}.
For the operators listed in Section~\ref{sec:oper} we use the following reasoning:

(c) and (wc) Compact and weakly compact operators are well known to be closed with respect to the norm convergence.

(cc) For completely continuous operators, see \cite[Chapter I., Exercise 9.3]{defant-floret}.
\item

(wpc) For weakly precompact operators, consider a sequence $(U_n)$ of  weakly precompact operators between Banach spaces $E$ and $F$ that converges in the norm to $U\in L(E,F)$.
Let $(x_j)\subset B_{E}$ be given. Inductively we construct  infinite sets $\en\supset N_1\supset N_2\supset\cdots$ such that $(U_nx_j)_{j\in N_n}$ is weakly Cauchy.
Let $M=\{m_1,m_2,m_3,\dots\}\subset \en$ be a diagonal set, i.e., we pick $m_1\in N_1$, $m_2\in N_2\cap (m_1,\infty)$, $m_3\in N_3\cap (m_2,\infty)$ and so on. We claim that $(Ux_m)_{m\in M}$ is weakly Cauchy. 

Indeed, let $y^*\in F^*$ and $\ep>0$ be given. We select $n\in \en$ such that $\norm{U-U_n}<\ep$ and let $m_0\in M$ be such that $\abs{y^*(U_nx_{m}-U_nx_{m'})}<\ep$ for $m,m'\in M$ with $m,m'\ge m_0$.
Then for $m,m'\in M$ with $m,m'\ge m_0$ we have
\[
\begin{aligned}
\abs{y^*(Ux_m)-y^*(Ux_{m'})}&\le \abs{y^*(Ux_m)-y^*(U_nx_m)}+\abs{y^*(U_nx_m)-y^*(U_nx_{m'})}+\\
&\quad+\abs{y^*(U_nx_{m'})-y^*(Ux_{m'})}\\
&\le \norm{y^*}\norm{U-U_n}+\ep+\norm{y^*}\norm{U_n-U}\le \ep(2\norm{y^*}+1).
\end{aligned}
\]
Hence $(Ux_m)_{m\in M}$ is weakly Cauchy and $U$ is weakly precompact.

(pwc) For pseudo weakly compact operators, consider a sequence $(U_n)$ of pseudo weakly compact operators between Banach spaces $E$ and $F$ that converges in the norm to $U\in L(E,F)$.
Let $(x_j)\subset E$ Right converge to $0$. Then $(x_j)$ is bounded by a number $M\ge 0$. Then for any $\ep>0$ there exists $n\in\en$ with $\norm{U_n-U}<\ep$. Let $j_0\in \en$ be such that $\norm{U_n(x_j)}<\ep$ for $j\ge j_0$. Then we have the following estimate for $j\ge j_0$:
\[
\norm{Ux_j}\le \norm{Ux_j-U_nx_j}+\norm{U_nx_j}<\ep(M+1).
\]
Hence $Ux_j\to 0$ in the norm and $U$ is pseudo weakly compact.

 (wcc) For weakly completely continuous operators, consider a sequence $(U_n)$ of weakly completely continuous operators between Banach spaces $E$ and $F$ that converges in the norm to $U\in L(E,F)$. Let $(x_j)\subset E$ be weakly Cauchy in $E$. Then $(x_j)$ considered as a sequence in $E\subset E^{**}$ weak$^*$-converges to some $x^{**}\in E^{**}$.
By the assumption, 
\[
z_n=U_n^{**}x^{**}=\text{weak}^*-\lim_{j\to \infty} U_n^{**}x_{j}
\]
is an element of $F$ for each $n\in\en$. Since 
\[
z=U^{**}x^{**}=\text{weak}^*-\lim_{j\to \infty} U^{**}x_{j}\in F^{**},
\]
and 
\[
\norm{z-z_n}=\norm{U^{**}x^{**}-U_n^{**}x^{**}}\le \norm{U-U_n}\norm{x^{**}},
\]
the vector $z$ is also contained in $F$. Hence $U$ is weakly completely continuous.

(Rcc) For Right completely continuous operators, consider a sequence $(U_n)$ of Right completely continuous operators between Banach spaces $E$ and $F$ that converges in the norm to $U\in L(E,F)$. Let $(x_j)\subset E$ be Right Cauchy in $E$. Then $(x_j)$ is bounded by some number $M>0$, and considered as a sequence in $E\subset E^{**}$ weak$^*$ converges to some $x^{**}\in E^{**}$.
By assumption, 
\[
z_n=U_n^{**}x^{**}=\text{weak}^*-\lim_{j\to \infty} U_n^{**}x_{j}
\]
is an element of $F$ for each $n\in\en$. Since 
\[
z=U^{**}x^{**}=\text{weak}^*-\lim_{j\to \infty} U^{**}x_{j}\in F^{**},
\]
and 
\[
\norm{z-z_n}=\norm{U^{**}x^{**}-U_n^{**}x^{**}}\le \norm{U-U_n}\norm{x^{**}},
\]
the vector $z$ is also contained in $F$. 

We have to prove that $Ux_j\to z$ in the Right topology. To this end we fix a weakly compact subset $L\subset F^*$  and denote by $\norm{\cdot}_{L}$ the supremum norm of a function on $L$. Let $N>0$ be such that $L\subset N\cdot B_{F^{*}}$. Given $\ep>0$, let $n\in\en$ be such that $\norm{U-U_n}<\ep$. In addition, we pick $j_0\in \en$ with $\norm{U_n^{**}x_j|_L-z_n|_L}_L<\ep$ for $j\ge j_0$. Then for each $j\ge j_0$ we have
\[
\begin{aligned}
\norm{U^{**}x_j|_L-z|_L}_L&\le \norm{U^{**}x_j|_L-U_n^{**}x_j|_L}_L+\norm{U_n^{**}x_j|_L-z_n|_L}_L+\norm{z_n|_L-z|_L}_L\\
&\le N\cdot M\cdot \norm{U^{**}-U_n^{**}}+\ep+N\cdot \norm{U^{**}-U_n^{**}}\norm{x^{**}}\\
&<\ep(N\cdot M+1+N\cdot M)=\ep(2NM+1).
\end{aligned}
\]
Hence $(Ux_j)$ converges to $z$ in the Right topology.

(uc) For unconditionally converging operators, see \cite[Theorem 2.8]{howard}.
    
\end{proof}

\section{Extensions of  operators on $\inxe$}
\label{s:extens}

We start this section by an extension formula for a strongly bounded operator on $\inxe$ to the space $C(B_{X^*},E)$. We will see in Theorem~\ref{t:ckeextense} that $S$ inherits properties of $T$.

\begin{thm}
\label{t:extense}
Let $T\colon \inxe\to F$ be a strongly bounded operator with the representing measure $G$. Then $T$ has a strongly bounded extension to $U\colon B(B_{X^*},E)\to F$ given by the formula
\[
Uf=\int_{B_{X^*}} f\di G,\quad f\in B(B_{X^*},E).
\]
The restriction $S=U|_{C(B_{X^*},E)}$ has the representing measure $G$.    
\end{thm}

\begin{proof}
Let $\lambda$ be a control measure for $G$ and let $f\in B(B_{X^*},E)$ be given. Since $G$ has values in $L(E,F)$, if $f$ is of the form $f=\sum_{i=1}^n \chi_{E_i}x_i$ for some $(E_i)\in \pi(B_{X^*})$ and $(x_i)\subset E$, then $\int_{B_{X^*}} f\di G=\sum_{i=1}^n G(E_i)x_i\in F$. Hence $U\in L(E,F)$ for each $f\in B(B_{X^*},E)$. We also have 
\[
y^{*}(Uf)=y^*\left(\int_{B_{X^*}} f\di G\right)=\int_{B_{X^*}} f\di G_{y^*},\quad f\in B(B_{X^*},E).
\]

Let $G'\colon\Borel(B_{X^*})\to L(E,F^{**})$ be the representing measure for $S=U|_{C(B_{X^*},E)}$.
Then for each $y^*\in F^*$ and $f\in C(B_{X^*},E)$ we have
\[
\int_{B_{X^*}} f\di G_{y^*}=y^*(Uf)=y^*(Sf)=(S^*y^*)(f)=\int_{B_{X^*}} f\di (G')_{y^*}.
\]
Hence $(G')_{y^*}=G_{y^*}$ for each $y^*\in F^*$, from which the equality $G=G'$ follows.
\end{proof}

\subsection{Preparatory lemmas}

Before we embark on the proof of the main result of this section, namely Theorem~\ref{t:ckeextense}, we need several auxiliary lemmas.

\begin{lemma}
    \label{l:wuc}
Let $K$ be a compact topological space. Then for a bounded sequence $(f_n)\in C(K,E)$ the following assertions hold:
\begin{enumerate}[(i)]
\item The sequence $(f_n)$ is weakly Cauchy if and only if $(f_n(t))$ is weakly Cauchy in $E$ for each $t\in K$.
\item The sequence $(f_n)$ weakly converges to $f\in C(K,E)$ if and only if $f_n(t)\to f(t)$ weakly for each $t\in K$.
\end{enumerate}    
\end{lemma}

\begin{proof}
The proof can be found, e.g., in \cite[Proposition 1.7.1 and Corollary 1.7.1]{cembranos-mendoza}.    
\end{proof}

\begin{lemma}
    \label{l:extense}
    Let $X$ be a separable $L_1$-predual and let $L\subset \ext B_{X^*}$ be a homogeneous compact set. Then there exists a  continuous affine homogeneous mapping $\varphi\colon B_{X^*}\to M_{\ahom}(L)$ such that $\varphi(t)=\hom \ep_t$ for $t\in L$ and $\norm{\varphi(t)}\le 1$, $t\in B_{X^*}$.
    \end{lemma}

\begin{proof}
Since $X$ is separable, $C(B_{X^*})$ is separable as well, and thus there exists an operator $L\colon M(B_{X^*})\to \ell_2$ that is injective and weak$^*$-to-norm continuous (see \cite[Lemma 1.6]{spu-bullscimath}). In particular, $L$ is a homogeneous affine homeomorphism of $B_{M(B_{X^*})}$ onto its image in $(\ell_2,\norm{\cdot})$.

We define a multivalued mapping $\Phi\colon B_{X^*}\to M_{\ahom}(L)$ by setting
\[
\Phi(t)=\begin{cases}
    \hom \ep_t,& t\in L,\\
    B_{M_{\ahom}(L)},& t\in B_{X^*}\setminus L.
\end{cases}
\]
We claim that $\Phi$ is a homogeneous lower semicontinuous affine mapping of $B_{X^*}$ to $B_{M(B_{X^*})}$. To verify this, we first observe that, for $t\in L$ and $\alpha\in\TT$,
we have 
\[
\begin{aligned}
\Phi(\alpha t)(f)&=\hom \ep_{\alpha t}(f)=\hom f(\alpha t)=\int_{\TT} \lambda^{-1}f(\lambda\alpha t)\di\lambda\\
&=\alpha\int_{\TT}\xi^{-1}f(\xi t)\di \xi
=\alpha (\hom f)(t)=(\alpha \hom \ep_t)(f),\quad f\in C(B_{X^*}).
\end{aligned}
\]
Hence $\Phi(\alpha t)=\alpha \Phi(t)$ for $t\in L$. For $t\in B_{X^*}\setminus L$,
\[
\Phi(\alpha t)= B_{M_{\ahom}(L)}=\alpha  B_{M_{\ahom}(L)}=\alpha \Phi(t).
\]

Further, $\Phi$ is lower semicontinuous, i.e., for any open set $U\subset M(B_{X^*})$ the set
\[
\Phi^{-1}(U)=\{t\in B_{X^*}\setsep \Phi(t)\cap U\neq\emptyset\}
\]
is open in $B_{X^*}$.
 Indeed, $\Phi$ is easily seen to be a continuous single-valued mapping on  $L$. Since $L$ is compact and $\Phi(t)\in B_{M_{\ahom}(L)}$ for $t\in L$, $\Phi$ is lower semicontinuous.

Finally, it is affine. Indeed, if $t_1,t_2\in B_{X^*}$ are distinct and $\alpha\in (0,1)$, then $\alpha t_1+(1-\alpha)t_2\notin \ext B_{X^*}$. Hence $\alpha t_1+(1-\alpha)t_2\notin L$, and thus
\[
\alpha \Phi(t_1)+(1-\alpha)\Phi(t_2)\subset \alpha B_{M_{\ahom}(L)}+(1-\alpha) B_{M_{\ahom}(L)}\subset B_{M_{\ahom}(L)}=\Phi(\alpha t_1+(1-\alpha)t_2).
\]

Thus $\Phi$ is a homogeneous lower semicontinuous affine mapping of $B_{X^*}$ to $B_{M(B_{X^*})}$. Then $L\circ \Phi$ is a homogeneous lower semicontinuous affine mapping of $B_{X^*}$ to the Banach space $\ell_2$ such that $L\circ \Phi$ has nonempty closed convex values in $\ell_2$. By \cite[Chapter 7, \S 22, Theorem 2]{lacey} (the real case) and \cite[Theorem 4.2]{olsen-sel} (the complex case), there exists a homogeneous continuous affine mapping $\varphi'\colon B_{X^*}\to \ell_2$ with $\varphi'(t)\in L(\Phi(t))$.
Then $\varphi(t)=L^{-1}(\varphi'(t))$ is a homogeneous continuous affine selection of $\Phi$. 
\end{proof}

\begin{lemma}
\label{l:faktorizaceK}
Let $G$ be the representing measure of a strongly bounded operator $T\colon\inxe\to F$ and let $S\colon C(B_{X^*},E)\to F$ be its extension given by $Sf=\int_{B_{X^*}} f\di G$, $f\in C(B_{X^*},E)$. Let $(f_n)\in B_{C(B_{X^*},E)}$ be a sequence and $\ep>0$.

Then there exists a separable $L_1$-predual $Y\subset X$ such that for the restriction $\pi\colon B_{X^*}\to B_{Y^*}$ the following assertions hold:
\begin{enumerate}[(i)]
    \item  For each $n\in\en$ there exists $f_n'\in C(B_{Y^*},E)$ with $f_n=f_n'\circ \pi$.
    \item There exists a sequence $(a_n')\subset B_{A_{\hom}(B_{Y^*},E)}=B_{Y\wh{\otimes}_\ep E}$ such that
    $\norm{Sf_n-T(a_n'\circ\pi)}<\ep$, $n\in\en$. 
\end{enumerate}
Moreover, if $(f_n)$ is weakly null (respectively, weakly Cauchy), then $(a_n'\circ \pi)$ is weakly null in $A_{\hom}(B_{X^*},E) $ (respectively, weakly Cauchy). If $\sum_{n=1}^\infty f_n$ is unconditionally converging, then the series $\sum_{n=1}^\infty (a_n'\circ\pi)$ is also unconditionally converging. 
\end{lemma}

\begin{proof}
Let $\lambda$ be a maximal invariant control measure for $G$. Using Lemma~\ref{l:separablredukce} we find an $L_1$-predual $Y\subset X$ such that, for the restriction $\pi\colon B_{X^*}\to B_{Y^*}$, the measure $\pi\lambda$ is maximal and (i) is satisfied. Then the operator $T'\colon A_{\hom} (B_{Y^*},E)\to F$ defined as $T'a'=T(a'\circ\pi)$, $a'\in A_{\hom} (B_{Y^*},E)$, has $\pi(G)$ as the representing measure with a control measure $\pi\lambda$, see Lemma~\ref{l:redukce}.
Let $(f_n')\subset C(B_{Y^*},E)$ be the functions given by (i).

Let $\delta>0$ be such that $\norm{\pi(G)}(E)<\frac{\ep}{2}$ whenever $\pi\lambda(E)<\delta$, $E\in \Borel(B_{Y^*})$. Since $B_{Y^*}$ is metrizable and $\pi\lambda$ is maximal, it is carried by $\ext B_{Y^*}$. Let $L'\subset \ext B_{Y^*}$ be a compact set with $(\pi\lambda)(\ext B_{Y^*}\setminus L')<\delta$. By taking $L=\bigcup_{\alpha\in\TT}\alpha L'$ we obtain a compact homogeneous set with the property $L\subset \ext B_{Y^*}$ and $(\pi\lambda)(\ext B_{Y^*}\setminus L)<\delta$.

By Lemma~\ref{l:extense}, we have a homogeneous affine continuous mapping $\varphi\colon B_{Y^*}\to B_{M_{\ahom}(L)}$ satisfying $\varphi(t)=\hom\ep_t$ for $t\in L$.
We set 
\[
a_n'(t)=\int_{L} f_n'\di \varphi(t),\quad t\in B_{Y^*}, n\in\en.
\]
Then each $a_n'$ is a homogeneous affine continuous function on $B_{Y^*}$  with values in $B_E$. Indeed, $a_n'$ is affine and homogeneous. Furthermore, it is continuous by \cite[Lemma 6.1(i)]{Kaspu-studia} and $\norm{a_n'}\le 1$ by \cite[Lemma 3.8(c)]{Kaspu-studia}.

Regarding the estimate, let $n\in\en$ be given and let $y^*\in B_{F^*}$ satisfy
\[
\norm{Sf_n-T(a_n'\circ\pi)}=\norm{Sf_n-S(a_n'\circ\pi)}=y^*(Sf_n-S(a_n'\circ\pi)).
\]
Then $\vhom (\pi(G))_{y^*}=(\pi(G))_{y^*}$, and thus 
\[
\begin{aligned}
\norm{Sf_n-T(a_n'\circ\pi)}&=y^*(Sf_n-S(a_n'\circ\pi))=y^*(S((f_n'-a_n')\circ \pi))\\
&=\int_{B_{Y^*}}(f_n'-a_n')\di (\pi(G)_{y^*})
=\int_{B_{Y^*}}(f_n'-a_n')\di (\pi(G)_{y^*})\\
&=\int_L(f_n'-a_n')\di (\pi(G)_{y^*})+\int_{\ext B_{Y^*}\setminus L}(f_n'-a_n')\di (\pi(G)_{y^*})\\
&\le \abs{\int_L(f_n'-a_n')\di \vhom (\pi(G)_{y^*})}+2\norm{\pi(G)}(\ext B_{Y^*}\setminus L)\\
&\le \abs{\int_L (\vhom f_n'-\vhom a_n')\di (\pi(G)_{y^*})}+\ep.
\end{aligned}
\]
For $t\in L$ we have
\[
\vhom a_n'(t)=a_n'(t)=\int_{L} f_n'\di \varphi(t)=\int_L f_n'\di (\hom\ep_t)=\ep_t(\vhom f_n')=\vhom f_n'(t).
\]
Hence
\[
\norm{Sf_n-T(a_n'\circ\pi)}\le \ep.
\]

If $(f_n)$ is weakly null, then $(f_n(t))_{n\in\en}$ is weakly null for each $t\in B_{X^*}$. Hence $(f_n'(t'))$ is weakly null for each $t'\in B_{Y^*}$, and thus $(f_n')\subset C(L,E)$ is weakly null. Since the mapping $J\colon C(L,E)\to A_{\hom}(B_{Y^*},E)$ given by the integration
\[
Jf'(t')=\int_L f'\di\varphi(t'),\quad f'\in C(L,E), t'\in B_{Y^*}
\]
is a bounded operator, $(a_n')$ is weakly null. Hence $(a_n'\circ\pi)$ is also weakly null.

The reasoning for the case $(f_n)$ is weakly Cauchy is similar.

If $\sum_{n=1}^\infty f_n$ is unconditionally converging, $\sum_{n=1}^\infty f_n'$ is unconditionally converging as well. Since $a_n'=J(f_n'|_L)$, $n\in\en$, the series $\sum_{n=1}^\infty (a_n'\circ\pi)$ is also unconditionally converging.
\end{proof}

\subsection{Properties of the extended operator}

Now we can show that the extended operator $S\colon C(B_{X^*},E)\to F$ inherits properties of an operator $T\colon \inxe\to F$.

\begin{thm}
\label{t:ckeextense}
Let $T\colon \inxe\to F$ be a strongly bounded operator with the representing measure $G$ and let $S\colon C(B_{X^*},E)\to F$ be its extension given by $Sf=\int_K f\di G$, $f\in C(B_{X^*},E)$.
\begin{enumerate}[(a)]
\item Let $\I$ denote any of the following category of operators between Banach spaces: compact, weakly compact, weakly precompact, completely continuous, weakly completely continuous, unconditionally converging.

Then $T$ belongs to $\I$ if and only if $S$ belongs to $\I$.
\item Let $\I$ be any of the following category of operators between Banach spaces: pseudo weakly compact, Right completely continuous.

Then $T$ belongs to $\I$ if $S$ belongs to $\I$ and the converse implication holds if $X$ is separable.
\end{enumerate}
\end{thm}

\begin{proof}
(a) The implication $\impliedby$ is easy as $T$ is the restriction of $S$.

For the implication $\implies$ we use the following scheme of the proof of the respective cases: We start with a sequence $(f_n)\in B_{C(B_{X^*},E)}$. Then for $\ep>0$ we find  using Lemma~\ref{l:faktorizaceK}  a sequence $(a_n)\in B_{A_{\hom}(B_{X^*},E)}$ with $\norm{S f_n-Sa_n}<\ep$ such that $(a_n)$ inherits properties of $(f_n)$.

(c) If $T$ is compact and $S$ is not, we can find a sequence $(f_n)\subset B_{C(B_{X^*},E)}$ and $\eta>0$ such that $\norm{Sf_n-Sf_m}>\eta$, $n,m\in\en$ are distinct. For $\ep=\frac{\eta}{3}$, let $(a_n)\subset B_{A_{\hom}(B_{X^*},E)}$ be the sequence with the properties above. Since $T$ is compact, we can assume that $(Ta_n)$ norm converges. But for $n,m\in\en$ distinct we have
\begin{equation}
\label{eq:tretina}
\begin{aligned}
\eta&<\norm{Sf_n-Sf_m}\le \norm{Sf_n-Sa_n}+\norm{Ta_n-Ta_m}+\norm{Sa_m-Sf_m}\\
&\le \frac{\eta}{3}+\norm{Ta_n-Ta_m}+\frac{\eta}3,
\end{aligned}
\end{equation}
which yields $\norm{Ta_n-Ta_m}>\frac{\eta}{3}$, a contradiction. Hence $S$ is compact.

(wc) If $T$ is weakly compact, we pick a sequence $(f_n)\subset B_{C(B_{X^*},E)}$.
For $\ep>0$, we take a sequence $(a_n)\subset B_{A_{\hom}(B_{X^*},E)}$ as above. Then
$B=T(\{a_n\setsep n\in\en\})$ is relatively weakly compact and 
\[
\{Sf_n\setsep n\in\en\}\subset B+\ep B_F.
\]
Hence $A$ is relatively weakly compact by \cite{de-blasi} and $S$ is weakly compact.

(wpc) If $T$ is weakly precompact, $S$ is its extension to $C(B_{X^*},E)$ and $(f_n)\subset B_{C(B_{X^*},E)}$, for $\ep_k=\frac1k$ we construct functions $(a_n^k)\subset B_{A_{\hom}(B_{X^*},E)}$ such that $\norm{Sf_n-Ta_n^k}\le \ep_k$, $n,k\in\en$.
Inductively we pick a sequence of infinite sets $\en\supset N_1\supset N_2\supset N_3\supset\cdots$
such that $(Ta_n^k)_{n\in N_k}$ is weakly Cauchy. We select an infinite diagonal set $M=\{m_1, m_2,m_3,\dots\}\subset \en$ such that $m_1\in N_1$, $m_2\in N_2\cap (m_1,\infty)$, $m_3\in N_3\cap (m_2,\infty)$ and so on. We claim that $(Sf_m)_{m\in M}$ is weakly Cauchy. 

Let $y^*\in F^*$ and $\ep>0$ be given. We select $k\in\en$ with $\ep_k=\frac1k<\ep$ and find $m_0\in M$ such that $\abs{y^*(Ta_m^k)-y^*(Ta_{m'}^k)}<\ep$ for $m,m'\in M$ with $m,m'\ge m_0$.
Then  for $m,m'\in M$ with $m,m'\ge m_0$ we have
\[
\begin{aligned}
\abs{y^*(Sf_m)-y^*(Sf_{m'})}&\le \abs{y^*(Sf_m)-y^*(Ta_m^k)}+\abs{y^*(Ta_m^k)-y^*(Ta_{m'}^k)}+\\
&\quad+\abs{y^*(Ta_{m'}^k)-y^*(Sf_{m'})}\\
&\le \norm{y^*}\ep_k+\ep+\norm{y^*}\ep_k\le \ep(2\norm{y^*}+1).
\end{aligned}
\]
Hence $S$ is weakly precompact.

(cc) If $T$ is completely continuous and $S$ is not, we pick a weakly null sequence $(f_n)\subset B_{C(B_{X^*},E)}$ such that for some $\eta>0$ we have $\norm{Sf_n-Sf_m}>\eta$, $n,m\in\en$ are distinct. Let $(a_n)\subset B_{A_{\hom}(B_{X^*},E)}$ be weakly null as above for $\ep=\frac{\eta}{3}$.
Then as in \eqref{eq:tretina} we see that $(Ta_n)$ is not norm convergent, a contradiction.

(wcc) If $T$ is weakly completely continuous and $(f_n)\subset B_{C(B_{X^*},E)}$ is a given weakly Cauchy sequence, let $y^{**}\in F^{**}$ be the weak$^*$ limit of the sequence $(Tf_n)\subset F\subset F^{**}$. Then for $\ep>0$, let $(a_n)\subset B_{A_{\hom}(B_{X^*},E)}$ be a weakly Cauchy sequence as above. Let $z\in F$ be the weak limit of $(Ta_n)$.
Then
\[
\norm{y^{**}-z}\le \liminf_{n\to\infty}\norm{Sf_n-Ta_n}\le\ep.
\]
Hence $\dist(y^{**},F)=0$, which gives $y^{**}\in F$ and $Sf_n\to y^{**}$ weakly.

(uc) If $T$ is unconditionally converging and $S$ is not, let $\sum_{n=1}^\infty f_n$ be an unconditionally converging series in $C(B_{X^*},E)$ such that $\sum_{n=1}^\infty Tf_n$ is not convergent. Using the Cauchy condition, we can assume that  $\norm{Tf_n}>\eta$, $n\in\en$, for some $\eta>0$. For $\ep=\frac{\eta}3$, let $\sum_{n=1}^\infty a_n$ be an unconditionally convergent series constructed as above. Then 
\[
\eta<\norm{Sf_n}\le \norm{Sf_n-Ta_n}+\norm{Ta_n}\le \norm{Ta_n}+\frac{\eta}3
\]
gives that $\sum_{n=1}^\infty Ta_n$ is not convergent. This contradiction gives the result that $S$ is unconditionally converging.

(b)  If $S$ is pseudo weakly compact  and $(a_n)$ is  a Right null sequence in $B_{A_{\hom}(B_{X^*},E)}$, then    $(a_n)$ is a Right null sequence in $C(B_{X^*},E)$, see \cite[Proposition 3.10]{kacena-jmaa}. Hence $(Ta_n)=(Sa_n)$ is norm convergent. Similarly, we argue for a Right completely continuous operator $S$. This shows the implication $\impliedby$.

Let $X$ be separable and $(f_n)\subset B_{C(B_{X^*},E)}$ be Right null (respectively Right Cauchy). Then we do not have to use the metrizable reduction in the proof of Lemma~\ref{l:faktorizaceK} and thus the constructed sequence $(a_n)\subset B_{A_{\hom}(B_{X^*},E)}$ is also Right null (respectively Right Cauchy). 

(pwc) If $T$ is pseudo weakly compact, the preceding remark and the same line of reasoning as in (cc) give that $S$ is pseudo weakly compact.

(Rcc) Let $T$ be Right completely continuous and $(f_n)\subset B_{C(B_{X^*},E)}$ be a Right Cauchy sequence. For $\ep>0$, let $(a_n)\subset B_{A_{\hom}(B_{X^*},E)}$ be a Right Cauchy sequence constructed above. Let $y^{**}\in F^*$ be the weak$^*$ limit of the sequence $(Sf_n)$ and $z\in F$ be the Right limit of $(Ta_n)$. Then $\norm{y^{**}-z}<\ep$, which gives that $y^{**}$ is in fact in $F$. 

To show that $Sf_n\to y^{**}$ in the Right topology, let $W\subset F^*$ be a weakly compact set with $W\subset M\cdot B_{F^*}$ for $M>0$ and $\ep>0$. Let 
$(a_n)\subset B_{A_{\hom}(B_{X^*},E)}$ be as usual. Then for the Right limit $z\in F$ of $(Ta_n)$ we have $\sup_{y^*\in W} \abs{y^*(Ta_n)-y^*(z)}\to 0$.
Since $\norm{Sf_n-Ta_n}<\ep$, we get $\norm{y^{**}-z}\le \ep$ and thus 
\[
\begin{aligned}
\sup_{y^*\in W} \abs{y^*(S f_n-y^{**})}&\le \sup_{y^*\in W}\abs{y^*(Sf_n)-y^*(Ta_n)}+\sup_{y^*\in W}\abs{y^*(Ta_n)-y^*(z)}+\\
&\quad +\sup_{y^*\in W}\abs{y^*(z)-y^*(y^{**})}\\
&\le M\cdot \ep+\sup_{y^*\in W}\abs{y^*(Ta_n)-y^*(z)}+M\cdot \ep,
\end{aligned}
\]
Since $\sup_{y^*\in W}\abs{y^*(Ta_n)-y^*(z)}\to 0$, we get that 
\[
\limsup_{n\to\infty}\sup_{y^*\in W} \abs{y^*(S f_n-y^{**})}\le 2M\ep.
\]
As $\ep>0$ is arbitrary, $Sf_n\to y^{**}$ in the Right topology.
\end{proof}

\begin{remark}
\label{r:compact}
Let $\I$ be a category of bounded operators between Banach spaces. If $T\colon \inxe\to F$ is a strongly bounded operator with the representing measure $G$, the property $G(B)\in \I$ for each $B\in\Borel (B_{X^*})$ does not guarantee $T\in \I$. In fact, an example constructed in \cite[Theorem 2.1]{saab-mpcps} provides an operator $T\colon C(2^\en,E)\to c_0$ such that the values of $G$ are compact, yet $T$ is not unconditionally converging.   
\end{remark}

\begin{ques}
\label{q:metriz}
We do not know whether (b) in Theorem~\ref{t:ckeextense} holds without the assumption of separability of $X$.
\end{ques}

\section{Corollaries for the space $\inxe$}

Now we can use Theorem~\ref{t:ckeextense} and the known theorems for the spaces $C(K,E)$ to obtain the results for the spaces $\inxe$. 

The main strategy of the proof is as follows: We consider an unconditionally converging operator $T\colon \inxe\to F$ and extend it to a strongly bounded operator $S\colon C(B_{X^*},E)\to F$ that inherits properties from $T$ (see Theorem~\ref{t:ckeextense}). Then we apply a known result for $C(B_{X^*},E)$ to deduce a behavior of $S$ and then its restriction $T$ shares the same property. 

By the described procedure, we can obtain the following corollaries.

\begin{cor}
    \label{c:dusledkyake}
    Let $X$ be an $L_1$-predual and $E$ be a Banach space.
    \begin{enumerate}[(a)]
        \item   If $E$ does not contain $\ell_1$, then $\inxe$  has the Dieudonné property.
        \item Let $E$ be a Banach space such that its each subspace has the property (V). Then $\inxe$ has the property (V).
        \item Let $E$ be a separable Banach space with the property (V). Then $\inxe$ has the property (V).
        \item Let $E$ be a Banach space with the property (u). Then $\inxe$ has the property semi-(V).
        \item  Let $E$ not contain $\ell_1$ and $T\colon \inxe\to F$ be strongly bounded. Then $T$ is weakly precompact.
       \item If $E$ is reflexive, then any strongly bounded operator $T\colon \inxe\to F$ is weakly compact.
             \item  If $E$ does not contain $\ell_1$ and $F$ does not contain $c_0$, then every operator $T\colon \inxe\to F$ is weakly precompact. 
       \item If $E$ does not contain $\ell_1$ and if $F$ is weakly sequentially complete,  then every operator $T\colon \inxe\to F$ is weakly compact. 
       \item If $E$ does not contain $\ell_1$ and if $F$ has the Schur property, then  every operator $T\colon \inxe\to F$ is  compact.
         \item  If $E^*$ has the Radon-Nikodým property, then every strongly bounded operator $T\colon \inxe\to F$ is weakly precompact. 
       \item If $E^*$ is separable, then every strongly bounded operator $T\colon \inxe\to F$ is weakly precompact. 
       \item  Let $E^*$ not contain a  copy of $\ell_1$ and let $T\colon \inxe\to F$ be strongly bounded.
    Then $T$ and $T^*$ are weakly precompact.
    \item If $E$ does not contain $c_0$ and $T\colon \inxe\to F$ is strongly bounded, then $T$ is unconditionally converging.
       
    \end{enumerate}
 
\end{cor}

\begin{proof}
(a) If $T\colon \inxe\to F$ is weakly completely continuous operator, it is unconditionally converging and thus admits a weakly completely continuous extension $S\colon C(B_{X^*},E)\to F$. Then \cite[Corollary 3]{ghenciu-bullpolish} (see also \cite{emmanuele} and \cite{kalton-saab-saab}) gives that $S$ is weakly compact. Hence $T$ is weakly compact, and $\inxe$ has the Dieudonné property.

(b) By \cite[Corollary 1.5 and the subsequent Remark]{rosenthal-ajm}, a Banach space $E$ has the property (V) hereditarily if and only if $E$ has the property (u) and does not contain $\ell_1$. 

So, let $T\colon \inxe\to F$ be an unconditionally converging operator. 
Then its extension $S\colon C(B_{X^*}, E)\to F$ is also unconditionally converging (see Theorem~\ref{t:ckeextense}). By \cite[Corollary 2.7]{ulger}, $C(B_{X^*},E)$ has the property (V). 
Hence $S$ is weakly compact and therefore $T$ is weakly compact as well. Hence $\inxe$ has the property (V).

(c) follows from \cite[Theorem 1]{randrian}.

(d) follows from \cite[Theorem 6]{kalton-saab-saab-bsm}.

(e) The operator $T$ admits a strongly bounded extension $S\colon C(B_{X^*},E)\to F$. Hence $S$ is weakly precompact by \cite[Corollary 2]{ghenciu-lewis}. Thus, also $T$ is weakly precompact.    

(f) Let $S\colon C(B_{X^*},E)\to F$ be a strongly bounded extension of $T$. By \cite[Corollary 4(i)]{ghenciu-lewis}, $S$ is weakly compact. Hence, also $T$ is weakly compact.

(g) Let $T\colon \inxe\to F$ be an operator. Since $F$ does not contain $c_0$, $T$ is unconditionally converging (see \cite[Theorem 5]{bessaga-pelczynski}). Hence $T$ is strongly bounded and admits a strongly bounded extension to $S\colon C(B_{X^*},E)\to F$. This extension is weakly precompact by \cite[Corollary 5(i)]{ghenciu-lewis}.

(h) This follows from \cite[Corollary 5(ii)]{ghenciu-lewis} and the assumption that $F$ is weakly sequentially complete.

(i) If $F$ has the Schur property, $T$ is, moreover, compact by (h).

(j) If $E^*$ has the Radon-Nikodým property, $E$ does not contain $\ell_1$ (see \cite[Theorem 5.7 and 2.34]{phelps-convex}). Hence, (e) yields the assertion.

(k) follows from (j) since $E^*$ has  the Radon-Nikodým property provided  $E^*$ is separable (see \cite[Theorem 2.12 and 5.7]{phelps-convex}). 

(l)  If $T\colon\inxe\to F$ is strongly bounded and $S\colon C(B_{X^*},E)\to F$ is its extension, both $S$ and $S^*$ are weakly precompact by \cite[Corollary 8]{ghenciu-lewis}. Hence $T$ is weakly precompact.
 Also, $T^*$ is weakly precompact.

 Indeed, if $(y_n^*)\subset B_{F^*}$ is such that $\mu_n=S^*y_n^*\in (C(B_{X^*},E))^*=M(B_{X^*},E^*)$ is weakly Cauchy, then
 \[
 \int_{B_{X^*}} a\di \mu_n=S^*y_n^*=\int_{B_{X^*}} a\di G_{y_n^*}=T^*y_n^*(a),\quad a\in \inxe.
 \]
 Hence, the sequence $(T^*y_n^*)=(\mu_n|_{\inxe})$ is also weakly Cauchy.
 
 (m) It is enough to consider the strongly bounded extension $S\colon C(B_{X^*},E)\to F$ and use \cite[Corollary 10]{ghenciu-lewis}.  
\end{proof}

 We remark that $A(K,E)$ has the Dunford-Pettis property, provided $E$ is either $L_1$-space or $E$ has the Schur property, see \cite[Corollary 2.3]{saab-rocky}.
 However, there is an example constructed in \cite{tal-israel}  of a Banach space $E$ such that $E^*$ possesses the Schur property (hence $E$ has the Dunford-Pettis property) and  $C([0,1],E)$ does not have the Dunford-Pettis property.
 
\begin{cor}
    \label{schur}
    Let $E$ have the Schur property. Then the following assertions hold:
    \begin{enumerate}[(a)]
    \item Every strongly bounded operator $T\colon \inxe\to F$ is completely continuous.
   
    \item If $F$ does not contain $c_0$, then every operator $T\colon \inxe\to F$ is completely continuous. 
    \item If $T\colon \inxe\to F$ is an operator with $T^*$ weakly precompact, then $T$ is completely continuous.
    \item If $T\colon \inxe\to F$ is an operator, then $T$ is weakly completely continuous if and only if $T$ is completely continuous.
            \end{enumerate}
\end{cor}

\begin{proof}
    (a) Consider the strongly bounded extension $S\colon C(B_{X^*},E)\to F$ and the apply \cite[Corollary 11(i)]{ghenciu-lewis}.

        (b) If $F$ does not contain $c_0$, an operator $T\colon \inxe\to F$ is unconditionally converging by \cite[Theorem 5]{bessaga-pelczynski}, and thus strongly bounded. By (a), $T$ is completely continuous.

    (c) By \cite[Corollary 2]{bator-lewis}, $T$ is unconditionally converging, and thus strongly bounded. Hence (a) yields the conclusion.

    (d) If $T$ is weakly completely continuous, $T$ is strongly bounded. Hence (a) gives that $T$ is completely continuous. The converse is obvious.
\end{proof}

\begin{cor}
    \label{ell1-uc}
    If $E$ does not contain $\ell_1$, then every unconditionally converging $T\colon \inxe\to F$ is weakly precompact and $T^*$ is unconditionally converging.
\end{cor}

\begin{proof}
The operator $T$ is strongly bounded, hence weakly precompact (see Corollary~\ref{c:dusledkyake}(e)) and has unconditionally converging adjoint (see \cite[Theorem 12]{ghenciu-lewis}).    
\end{proof}

\begin{cor}
\label{c:ell1-compl}
Let $E$ not contain a complemented copy of $\ell_1$. Then any operator $T\colon\inxe\to F$ has an unconditionally converging adjoint.    
\end{cor} 

\begin{proof}
Let $T\colon \inxe\to F$ be an operator whose adjoint $T^*$ is not unconditionally converging. By  \cite[Problem 8, p.\,54]{diestel-sequences}, there exists an isomorphic copy $G$ of $c_0\subset F^*$ such that $T^*|_G$ is an isomorphism. Let $L\colon c_0\to G$ be an isomorphism of $c_0$ onto $G$, then $c_0$ embeds into $(\inxe)^*=M_{\bnd,\ahom}(B_{X^*},E^*)$. Hence $c_0$ embeds into $M(B_{X^*},E^*)$, and thus $c_0$ embeds into $E^*$ by the proof of \cite[Theorem 1]{saab-pams}. Thus $E$ contains a complemented copy of $\ell_1$, see \cite[Theorem 4]{bessaga-pelczynski}.
\end{proof}

\section{Operators on $\inxe$, where $\ext B_{X^*}$ is scattered}
\label{s:scatt}

For spaces $C(L,E)$, where $L$ is a scattered compact space (we recall that a topological space $L$ is \emph{scattered} if every nonempty subset of $L$ has an isolated point),  it is possible to deduce the properties of $C(L,E)$ from the properties of $E$, see \cite{bombal-cembranos}, \cite{cembranos-bullaus} and \cite{kacena-jmaa}. To get such results, it is enough to have an operator with a discrete control measure (we recall that a measure $\lambda\in M^+(L)$ is \emph{discrete} if there exists a countable set $\{t_n\setsep n\in\en\}$ in $L$ and a summable sequence $(a_n)$ of positive numbers with $\mu=\sum_{n=1}^\infty a_n \ep_{t_n}$).
This case appears when the compact space $L$ is countable. Then the compact convex set $K=M^1(L)$ is a Bauer simplex with $\ext K=L$. Then $C(L,E)=A(K,E)$. Thus, it is natural to consider the case where our simplex $K$ has only countably many extreme points. We remark that in that case  $K$ is metrizable, see \cite[Theorem 10.56]{lmns}. 

In case $K$ is metrizable, $K$ is a \emph{standard} compact convex set, i.e., $\ext K$ is $\mu'$-measurable for any completion $\mu'$ of a measure $\mu\in M^1(K)$ and any maximal measure $\mu$ on $K$ is carried by $\ext K$ in the sense that $\mu'(K\setminus\ext K)=0$. We will formulate our results for an $L_1$-predual $X$, whose dual unit ball $B_{X^*}$ is a standard compact convex set and $\ext B_{X^*}$ is scattered.

Before we proceed to the positive results, we present a classical example showing that in a non metrizable case the fact that $\ext K$ is scattered  does not guarantee a discrete control measure in general.

\begin{example}
    \label{ex:dikobraz} Let $\ef=\er$.
For each compact space $L$ that is not scattered there exists a simplex $K$ such that:
\begin{enumerate}[(a)]
    \item $\ext K$ is scattered;
    \item there exists an operator $T\colon A(K)\to\er$ such that its representing measure does not have a discrete control measure;
    \item for each Banach space $E$, the space $C(L,E)$ is $1$-complemented in $A(K,E)$.
\end{enumerate}
\end{example}

\begin{proof}
Given a compact space $L$ that is not scattered, let 
$M=L\times\{-1,0,1\}$ with the porcupine topology (see \cite[Proposition I.4.15]{alfsen}). This means that the points of $L\times\{-1,1\}$ are isolated and basis of neighborhoods of $(t,0)\in L\times\{0\}$ are sets of the form
\[
(U\times\{-1,0,1\})\setminus \{(t,-1),(t,1)\},\quad U\subset L\text{ a neighborhood of }t.
\]
Then $M$ is a compact Hausdorff space.
Let
\[
\H=\{f\in C(M,\er)\setsep f(t,0)=\frac12(f(t,-1)+f(t,1)),t\in L\}.
\]
If $K$ is the \emph{state space} of $\H$, i.e., 
\[
K=\{\omega\in  \H^*\setsep \omega(1)=\norm{\omega}=1\},
\]
with the weak$^*$ topology, then it is well known (see \cite[Proposition I.4.15]{alfsen}) that $K$ is a simplex with $\ext K$ homeomorphic to $M\setminus (L\times\{0\})$ by the evaluation mapping $\phi\colon M\to K$. Hence $\ext K$ is scattered.

Since $L$ is not scattered, there exists a continuous  measure $\lambda\in M^1(L)$ (i.e, $\lambda(\{t\})=0$ for each $t\in L$). Then the mapping $T\colon A(K)\to \er$ defined as $Ta=\int_K a\di\phi(\lambda)$ has as the representing measure the maximal measure $\phi\lambda$ (see \cite[Lemma 14.2(1)]{kal-ron-spu-rend}). This measure satisfies $(\phi\lambda)(\ext K)=0$. Nevertheless, the measure $\phi\lambda$ does not have a discrete maximal control measure: such a measure is carried  by $\ext K$, and this is impossible for $\phi\lambda$.

If $E$ is a Banach space, the space $A(K,E)$ is isometric to 
\[
\H_E=\{f\in C(M,E)\setsep f(t,0)=\frac12(f(t,-1)+f(t,1)),t\in L\}
\]
via the mapping $J\colon A(K,E)\to \H_E$ given as $f\mapsto f\circ\phi$.
Indeed, any function $f\in A(K,E)$ is continuous on $\phi(M)$ and satisfies the condition $f(\phi(t,0))=\frac12(f(\phi(t,1)+f(\phi(t,-1)))$, $t\in L$. Thus $f\circ \phi\in \H_E$.

On the other hand, if $f\in \H_E$ is given, then $\wt{f}\colon M^1(M)\to E$ defined as $\wt{f}(\mu)=(B)\-\int_M f\di\mu$, $\mu\in M^1(M)$, is an affine continuous function from $M^1(M)$ to $E$. Let $\pi\colon M^1(M)\to K$ be the restriction mapping. Given $\mu_1,\mu_2\in M^1(M)$ with $\pi(\mu_1)=\pi(\mu_2)$, i.e., $\mu_1-\mu_2\in \H^\perp$, then we claim that $\wt{f}(\mu_1)=\wt{f}(\mu_2)$. Let $x^*\in E^*$ be arbitrary. Then $x^*\circ f\in \H$, and thus
\[
\int_M x^*\circ f\di\mu_1=\int_M x^*\circ f\di\mu_2.
\]
Since $x^*\in E^*$ is arbitrary, the equality $\wt{f}(\mu_1)=\wt{f}(\mu_2)$ follows. 
Thus, there is a  unique function $f'\in A(K,E)$ satisfying $\wt{f}=f'\circ \pi$. Then $f=f'\circ\phi$ and $J$ is a surjective isometry.

Further, $C(L,E)$ is isometric to
\[
\H_L=\{f\in C(M,E)\setsep f(t,0)=f(t,-1)=f(t,1)),t\in L\}.
\]
Then the mapping $P\colon \H_E\to \H_L$ defined as
\[
Pf(t,i)=f(t,0),\quad t\in L, i\in\{-1,0,1\},
\]
is a projection onto $\H_L$ of norm $1$.
\end{proof}

Before the proof of the main result section of this section, namely Theorem~\ref{t:discrete}, we need a lemma.

\begin{lemma}
    \label{l:discrete-jedna}
 Let $T\colon \inxe\to F$ be  a strongly bounded operator with the representing measure $G$ and a discrete control maximal measure $\lambda$.
 Consider a category $\I$ of operators between Banach spaces from the following list: compact, weakly compact, weakly precompact, completely continuous, pseudo weakly compact, weakly completely continuous, Right completely continuous, unconditionally converging.
 
Then, the following assertions are equivalent:
\begin{enumerate}[(i)]
    \item The operator $T$ belongs to $\I$.
    \item The operator $G(B)$ belongs to $\I$ for each $B\in\Borel(B_{X^*})$.
\end{enumerate}
 \end{lemma}

 \begin{proof}
(i)$\implies$(ii) follows from Corollary~\ref{c:operbecka}.

(ii)$\implies$(i) Let $\lambda=\sum_{n=1}^\infty a_n\ep_{t_n}$, where $H=\{t_n\setsep n\in\en\}\subset \ext B_{X^*}$ and $(a_n)$ is a summable sequence of positive numbers. We observe that each measure $G_{y^*}$ is carried by $H$.
For each $n\in\en$ we consider the operator $T_n\colon \inxe\to F$ defined as
\[
T_na=G(\{t_n\})a(t_n),\quad a\in \inxe =A_{\hom} (B_{X^*},E).
\]
Then it is easy to see that each $T_n$ belongs to $\I$. 
(The only categories that deserve a closer attention are the ones of pseudo weakly compact and Right completely continuous operators. But if $(a_n)\subset \inxe$ is Right convergent (respectively, Right Cauchy), then $(a_n(t))$ is Right convergent (respectively, Right Cauchy) in $E$ for each $t\in B_{X^*}$.)

Thus, the partial sums $S_m=\sum_{n=1}^m T_n\in\I$ for each $m\in\en$. To conclude the proof it is enough to show that $\norm{T-S_m}\to 0$. So, let $\ep>0$ be given. We find $\delta>0$ such that $\norm{G}(B)<\ep$ whenever $B\in\Borel(B_{X^*})$ satisfies $\lambda(B)<\delta$. Let $m_0\in \en$ be such that $\sum_{n=m_0}^\infty a_n<\delta$ and fix arbitrary $a\in B_{\inxe}$ and $y^*\in B_{F^*}$. Then 
for $m\ge m_0$ we have
\[
\lambda(H\setminus\{t_1,\dots, t_m\})=\sum_{n=m+1}^\infty a_n<\delta,
\]
and thus $\norm{G}(H\setminus\{t_1,\dots, t_m\})<\ep$.
Since $G(B)^*y^*=G_{y^*}(B)$ for each $B\in\Borel(B_{X^*})$, we get
\[
\begin{aligned}
 &\abs{y^*(Ta)-y^*(S_ma)}=\abs{(T^*y^*)(a)-y^*(\sum_{n=1}^m G({t_n})a(t_n))}\\
 &=
 \abs{\int_{B_{X^*}} a\di G_{y^*}-\sum_{n=1}^m \left(G(\{t_n\})^*y^*\right)(a(t_n))}\\
 &\le\abs{\int_{\{t_1,\dots, t_m\}} a\di G_{y^*}-\sum_{n=1}^m \left(G(\{t_n\})^*y^*\right)(a(t_n))}+\abs{\int_{H\setminus \{t_1,\dots, t_m\}} a\di G_{y^*}}\\
 &=\abs{\sum_{n=1}^m G_{y^*}(\{t_n\})a(t_n)-\sum_{n=1}^m \left(G(\{t_n\})^*y^*\right)(a(t_n))}
 +\norm{a}\cdot \norm{G}(H\setminus \{t_1,\dots,t_m\})\\
 &\le 0+1\cdot \ep.
\end{aligned}
\]
Hence 
\[
\norm{T-S_m}=\sup_{y^*\in B_{F^*}}\sup_{a\in B_{\inxe}}\abs{y^*(Ta)-y^*(S_ma)}\le \ep,\quad m\ge m_0.
\]
Since the category $\I$ is closed with respect to the operator norm, the conclusion follows.
 \end{proof}

\begin{thm}
    \label{t:discrete}
    Let $X$ be an $L_1$-predual such that $B_{X^*}$ is a standard compact convex set and $\ext B_{X^*}$ is scattered (in particular, if $\ext B_{X^*}$ is countable or compact scattered).
    Consider a property $\P$ from the following list: property (V), property semi-(V), Dunford-Pettis property, reciprocal Dunford-Pettis property, Dieudonné property, Right Dieudonné property, to be sequentially Right.
    
    Then $\inxe$ has $\P$ if and only if $E$ has $\P$. 
\end{thm}

\begin{proof}
Obviously, $E$ has the property $\P$ is a necessary condition because $E$ has a $1$-complemented copy in $\inxe$.

Indeed, we find $x_0\in S_X$ and $x_0^*\in S_{X^*}$ such that $x_0^*(x_0)=1$. Let $J\colon E\to \inxe$ be given as $Je=x_0\otimes e$, $e\in E$. Then $J$ is an isometric embedding and an operator $P\colon \inxe=A_{\hom}(B_{X^*},E)\to \Rng J$ defined by
\[
Pf=x_0\otimes f(x_0^*),\quad f\in A_{\hom}(B_{X^*},E),
\]
is a projection onto $\Rng J$ of norm $1$.

For the proof of sufficiency, we use Lemma~\ref{l:discrete-jedna} and Corollary~\ref{c:operbecka}.  So, let us assume that $E$ has the property $\P$.

We present the proof scheme for the case where $\P$ is the property (V). Let $T\colon\inxe\to F$ be unconditionally converging. Then it is strongly bounded with the representing measure $G$ whose control measure $\lambda$ is maximal.  

We know that $\ext B_{X^*}$ is universally measurable. Further, any maximal measure is carried by $\ext B_{X^*}$. Since this set is scattered, the measure $\lambda$ is discrete (see \cite[Exercise 439X(h)]{fremlin4}). Now, Corollary~\ref{c:operbecka} gives that $G(B)\colon E\to F$ is unconditionally converging for each $B\in \Borel(B_{X^*})$. By the assumption, every $G(B)$ is weakly compact. Lemma~\ref{l:discrete-jedna} yields $T$ is weakly compact. Hence $\inxe$ has the property (V)

The remaining cases are analogous.
\end{proof}

\section*{Acknowledgement}

We are grateful to Professor Jakub Rondoš for several comments which helped us to improve the presentation of the paper.

\section*{Declarations}

The authors have no relevant financial or non-financial interests to disclose.
No funding was received to assist with the preparation of this manuscript.
On behalf of all authors, the corresponding author states that there is no conflict of interest. No data were generated or analyzed in this study.

\bibliographystyle{acm}
\bibliography{affine}

\end{document}